\documentclass[journal]{IEEEtran}
\usepackage{amssymb}

\usepackage[pdftex]{graphicx}
\usepackage{tikz}
\usepackage{amsmath}
\usepackage{framed}
\usepackage{bm}
\usepackage{cite}
\usepackage{comment}
\usepackage{algorithmic}
\usepackage{algorithm}
\usepackage{url}
\usepackage{booktabs}
\usepackage{array}

\usepackage{mathtools}
\mathtoolsset{showonlyrefs=true} 

\usepackage{amsfonts}
\usepackage{color}
\usepackage{cases}
\usepackage{enumerate}

\newtheorem{theorem}{Theorem}

\newtheorem{remark}{Remark}
\newtheorem{lemma}{Lemma}
\newtheorem{corollary}{Corollary}

\newtheorem{proposition}{Proposition}

\newtheorem{proof}{Proof}

\usepackage{blkarray}

\newcommand{\paren}[1]{\left(#1\right)}

\DeclareMathOperator{\tr}{tr}

\newcommand{\e}{\mathrm{e}}
\newcommand{\E}{\mathrm{e}}
\newcommand{\D}{\mathrm{d}}
\newcommand{\diff}{\mathrm{d}}

\newcommand{\Real}{\mathbb{R}}

\newcommand{\bb}{\mathbb}

\usepackage{blkarray}

\usepackage{latexsym}
\def\qed{\hfill $\Box$}

\usepackage{cite}
\usepackage{amsmath}
\usepackage{url}

\begin{document}
%
\title{Uniqueness Analysis of Controllability Scores and\\ Their Application to Brain Networks}
\author{Kazuhiro Sato and Ryohei Kawamura\thanks{K. Sato and R. Kawamura are with the Department of Mathematical Informatics, Graduate School of Information Science and Technology, The University of Tokyo, Tokyo 113-8656, Japan, email: kazuhiro@mist.i.u-tokyo.ac.jp (K. Sato), 04ryohei@gmail.com (R. Kawamura) }}
\maketitle
\thispagestyle{empty}
\pagestyle{empty}

\begin{abstract}
Assessing centrality in network systems is critical for understanding node importance and guiding decision-making processes. In dynamic networks, incorporating a controllability perspective is essential for identifying key nodes. In this paper, we study two control theoretic centrality measures---the Volumetric Controllability Score (VCS) and Average Energy Controllability Score (AECS)---to quantify node importance
in linear time-invariant network systems. We prove the uniqueness of VCS and AECS for almost all specified terminal times, thereby enhancing their applicability beyond previously recognized cases. This ensures their interpretability, comparability, and reproducibility.
Our analysis reveals substantial differences between VCS and AECS in linear systems with symmetric and skew-symmetric transition matrices. 
We also investigate the dependence of VCS and AECS on the terminal time and prove that when this parameter is extremely small, both scores become essentially uniform.
Additionally, we prove that a sequence generated by a projected gradient method for computing VCS and AECS converges linearly to both measures under several assumptions.
 Finally, evaluations on brain networks modeled via Laplacian dynamics using real data reveal contrasting evaluation tendencies and correlations for VCS and AECS, with AECS favoring brain regions associated with cognitive and motor functions, while VCS emphasizes sensory and emotional regions.
\end{abstract}

\begin{IEEEkeywords}
Brain networks, centrality, controllability scores, linear network systems
\end{IEEEkeywords}

\IEEEpeerreviewmaketitle

\section{Introduction} \label{sec:intro}
Assessing centrality in network systems is crucial for understanding the relative importance of nodes and identifying key influencers, informing decision-making in various contexts such as social networks, infrastructure systems, and brain networks \cite{cimpeanu2023social, fang2016resilience, faramondi2020multi, kong2022influence, sporns2010networks, van2013network}. 
The choice of centrality measure varies depending on the type of network system and the specific aspects under investigation \cite{bloch2023centrality, rodrigues2019network}.
However, static measures alone cannot capture how interventions propagate over time.
In dynamic network systems, incorporating a controllability perspective is essential \cite{kalman1960general, liu2011controllability, summers2015submodularity}.
We therefore ask not only “which nodes sit at the heart of the network?” but also “which nodes allow us to steer the system most effectively?” \cite{amani2018controllability, bof2016role, gu2015controllability, jalili2015optimal, liu2012control, lindmark2021centrality, pasqualetti2014controllability}.

To answer this, Sato and Terasaki \cite{sato2022controllability} introduced two control theoretic centrality measures---Volumetric Controllability Score (VCS) and Average Energy Controllability Score (AECS)---for linear network systems
\begin{align}
\dot{x}(t) = Ax(t), \label{system0}
\end{align}
where
$A=(a_{ij})\in \Real^{n\times n}$ encodes the network structure; the elements $x_1,\ldots, x_n$ of state $x=(x_i)\in \Real^n$ represent nodes, and $a_{ij}\neq 0$ indicates a directed edge
 from state node $x_j$ to state node $x_i$.
VCS and AECS are defined as solutions to a convex optimization problem referred to as the controllability scoring problem.
To calculate them, an algorithm employing the projected gradient method onto the standard simplex has been proposed in \cite{sato2022controllability} with a theoretical convergence guarantee.

\begin{table*}[t]
\centering
\caption{Summary of Tendencies and Correlations for AECS and VCS in Evaluating Brain Regions}
\label{table:summary}
\begin{tabular}{>{\raggedright}p{3cm} p{12cm}}
\toprule
\textbf{Metric} & \textbf{Tendencies and Correlations} \\
\midrule
\textbf{AECS} & \begin{itemize}
    \item Tends to highly evaluate brain regions associated with cognitive function and motor control.
    \item Tends to lowly evaluate brain regions associated with sensory processing and emotional regulation.
    \item Exhibits moderately positive correlations with conventional centrality metrics:
    \begin{itemize}
        \item Indegree: $\approx 0.64$
        \item Outdegree: $\approx 0.64$
        \item Betweenness: $\approx 0.68$
        \item PageRank: $\approx 0.64$
    \end{itemize}
    \item Shows a moderately negative correlation with Average Controllability ($\approx -0.60$).
\end{itemize} \\
\midrule
\textbf{VCS} & \begin{itemize}
    \item Tends to highly evaluate brain regions associated with sensory processing and emotional regulation.
    \item Tends to lowly evaluate brain regions associated with cognitive function and motor control.
    \item Exhibits weak negative correlations with conventional centrality metrics:
    \begin{itemize}
        \item Indegree: $\approx -0.30$
        \item Outdegree: $\approx -0.27$
        \item Betweenness: $\approx -0.26$
        \item PageRank: $\approx -0.29$
    \end{itemize}
    \item Shows a strong positive correlation with Average Controllability ($\approx 0.84$).
\end{itemize} \\
\bottomrule
\end{tabular}
\end{table*}

The uniqueness of VCS and AECS is crucial for their use as centrality measures for each state node, ensuring interpretability, comparability, and reproducibility.
The uniqueness has been rigorously proven for asymptotically stable system \eqref{system0} in \cite{sato2022controllability}.
However, for unstable systems, uniqueness for any specified time parameter $T$ has only been guaranteed 
in some special cases, leaving several aspects of the theoretical analysis incomplete.

Therefore, the primary objective of this paper is to address several unresolved issues in \cite{sato2022controllability}.
First, it remains unclear whether VCS and AECS can be uniquely defined for all linear systems of the form \eqref{system0}, particularly for Laplacian dynamics associated with directed graphs. 
Notably, Laplacian dynamics find applications in multiagent systems and opinion dynamics in social networks \cite{clark2012leader, fitch2016joint, pirani2018robustness, she2021energy, yi2021shifting}.
Second, the existing results provided in the literature do not specify the exact class of systems for which VCS and AECS differ. Third, the effect of time parameter $T$ on VCS and AECS has not been adequately investigated. Fourth, the convergence analysis of the algorithm for computing VCS and AECS proposed in \cite{sato2022controllability} remains insufficient. 
Finally, the effectiveness of VCS and AECS in real-world network systems has yet to be fully demonstrated.

 The contributions of this paper are as follows:
\begin{itemize}
    \item  We proved that, for any given matrix $A$ in system \eqref{system0}, VCS and AECS are unique for almost all specified time parameters $T$. This result extends the applicability of these centrality measures to a much broader class of systems.

\item We proved that system \eqref{system0} yields significant differences between VCS and AECS when $A$ is symmetric, while yielding identical values when $A$ is skew-symmetric.

\item  We investigated the dependence of VCS and AECS on the specified time parameter $T$. In particular, we proved that when $T$ is extremely small, VCS and AECS yield essentially uniform scores. Moreover, this result was confirmed by numerical experiments.

\item We conducted a detailed convergence analysis of the algorithm proposed in \cite{sato2022controllability} for computing VCS and AECS, and proved that, under several assumptions, a sequence generated by the algorithm converges linearly to either VCS or AECS. This finding was also validated by numerical experiments.

\item 
In our evaluation of brain networks modeled via Laplacian dynamics using real data, we observed distinct evaluation tendencies and correlations for VCS and AECS, as summarized in Table~\ref{table:summary}. Specifically, the AECS metric tends to assign higher values to brain regions associated with cognitive function and motor control, while lower values are attributed to regions involved in sensory processing and emotional regulation. In contrast, VCS shows the opposite trend by highly evaluating sensory and emotional regions and assigning lower values to cognitive and motor regions. Furthermore, AECS exhibits moderately strong positive correlations with traditional centrality measures (indegree, outdegree, betweenness, and PageRank) \cite{bloch2023centrality, rodrigues2019network} and a moderately negative correlation with Average Controllability as introduced in \cite{summers2015submodularity}, whereas VCS demonstrates weak negative correlations with these conventional metrics but a strong positive correlation with Average Controllability. 
\end{itemize}

The remainder of this paper is organized as follows.
In Section \ref{sec:Prelim}, we define VCS and AECS, and summarize the existing results presented in \cite{sato2022controllability}.
In Section \ref{sec:uniqueness}, we discuss the uniqueness of VCS and AECS for any matrix $A\in {\bb R}^{n\times n}$.
In Section \ref{sec:special}, we clarify the class of systems that yield differences between VCS and AECS.
In Section \ref{Sec_effect_T}, we investigate the dependence of VCS and AECS on the specified time parameter
$T$. 
 In Section \ref{Sec_Convergence}, we present a detailed convergence analysis of the algorithm for computing VCS and AECS. 
In Section \ref{Sec_numerical}, we evaluate VCS and AECS in brain networks modeled as Laplacian dynamics defined by using real data.
Finally, Section \ref{Sec_conclusion} concludes this paper.

{\it Notation:}
 The sets of real and complex numbers are denoted by ${\bb R}$ and ${\bb C}$, respectively.
 For $a\in {\bb C}$, ${\rm Re}(a)$ denotes the real part of $a$.
 For a matrix $X\in {\bb R}^{m\times n}$, $X^{\top}$ denotes the transpose of $X$.
 For a square matrix $A\in {\bb R}^{n\times n}$, $\det A$ and
${\rm tr}(A)$ denote the determinant and  diagonal sum of $A$, respectively.
For a symmetric matrix $A$,
$A\succ O$ denotes the positive definite matrix, where $O$ is the zero matrix.
The symbol $I$ denotes the identity matrix of appropriate size.
Given a vector $v=(v_i)\in {\bb R}^n$, $\|v\|$ and ${\rm diag}(v_1,\ldots, v_n)$ denote the usual Euclidean norm $\|v\|=\sqrt{v^{\top}v}$
and the diagonal matrix with the diagonal elements $v_1,\ldots, v_n$, respectively.
Instead of ${\rm diag}(v_1,\ldots, v_n)$, we also use ${\rm diag}(v)$.
The symbol ${\bf 1}$ represents a column vector whose elements are all $1$.
The symbol $S_n$ denotes the symmetric group of order $n$.
For $\sigma \in S_n$, ${\rm sgn}(\sigma)$ denotes the sign of the permutation $\sigma$.

\section{Preliminaries} \label{sec:Prelim}

To define two controllability scores, VCS and AECS,
we introduce a virtual system that differs from \eqref{system0}:
\begin{align}
  \dot x(t) = Ax(t) + {\rm diag}(\sqrt{p_1},\ldots, \sqrt{p_n})u(t),\label{eq:lti}
\end{align}
which establishes a one-to-one correspondence between each state node $x_i$ and a virtual input node $u_i$, thereby allowing us to associate the value $p_i$ with $x_i$.
For example, consider system \eqref{system0} with the network structure illustrated in Fig.\,\ref{fig:directed1}.
To define controllability scores for the system,
we assume---as illustrated in Fig.\,\ref{fig:CS_idea}---that there is a one-to-one correspondence between state nodes and virtual input nodes, with nonnegative weights $p_1,\ldots, p_{10}$ assigned to the directed edges from the input nodes to the state nodes. Then, the basic idea behind the controllability score is to interpret the values $p_1,\ldots, p_{10}$ (obtained when controllability is maximized) as the importance of the state nodes \textcircled{\scriptsize 1},\ldots, \textcircled{\scriptsize 10} in Fig.\,\ref{fig:directed1}, respectively.

\begin{figure}[t]
    \centering
    \includegraphics[width=4.5cm]{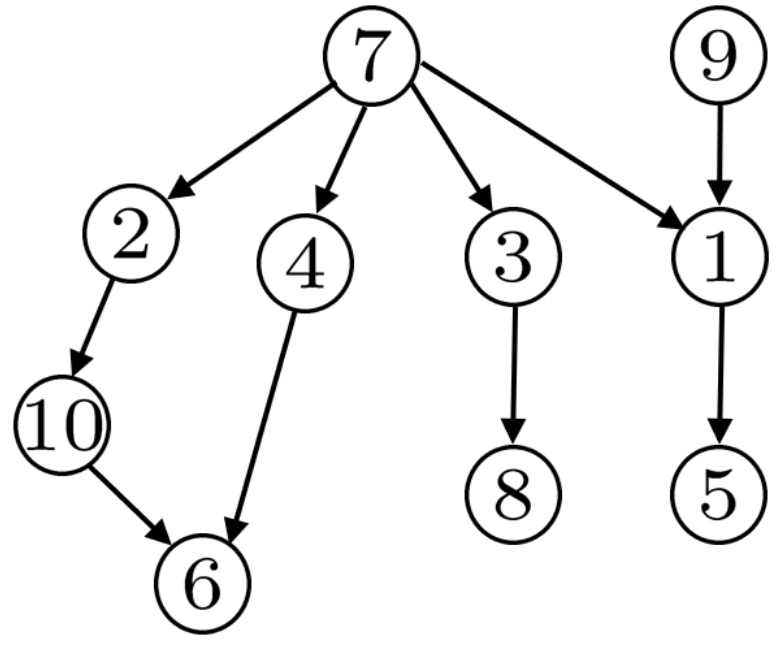}
    \caption{Directed graph with a hierarchical structure. State node \textcircled{\scriptsize 7} is expected to be more important in terms of controllability than state nodes \textcircled{\scriptsize 5}, \textcircled{\scriptsize 6}, and \textcircled{\scriptsize 8}. This is evident because the input associated with state node \textcircled{\scriptsize 7} influences all state nodes except for \textcircled{\scriptsize 9} via directed edges, whereas the inputs associated with state nodes \textcircled{\scriptsize 5}, \textcircled{\scriptsize 6}, and \textcircled{\scriptsize 8} do not affect any other state nodes.}
    \label{fig:directed1}
\end{figure}

\begin{figure}[t]
    \centering
    \includegraphics[width=8cm]{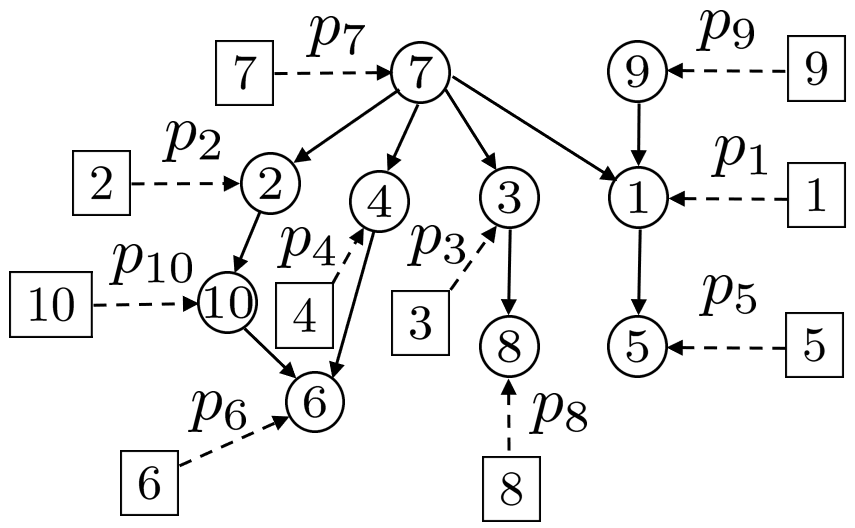}
    \caption{Idea for defining controllability scores.
The circle (\(\circ\)) represents a state node, and the square (\(\square\)) represents a virtual input node.
 Each state node corresponds one-to-one with an input node.}
    \label{fig:CS_idea}
\end{figure}

Then, for any positive number $T$, we define two convex sets on ${\bb R}^n$:
\begin{align}
    X_T &:= \{p\in {\bb R}^n \mid W(p,T)\succ O\}, \\
     \Delta & := \left\{ p = (p_i) \in \Real^n \left|
    \begin{array}{l}
         \sum_{i=1}^n p_i = 1,    \\
         0 \leq p_i \quad (i=1,\dots,n)  
    \end{array}
    \right. \right\}, \label{eq:simplex}
\end{align}
where $W(p,T)$ is the finite-time controllability Gramian of system \eqref{eq:lti}, and is given by
\begin{align}
    W(p, T)  
     = \sum_{i=1}^n p_i W_i(T) \label{Def_Wc}
\end{align}
with
\begin{align}
\label{eq:finite_time_horizon_gramian}
    W_i(T) := \int_0^T \exp(At) e_ie_i^\top \exp(A^\top t) \ \D t.
\end{align}
Here, $e_i\in\Real^n$
 denotes the standard vector which has 1 at $i$-th position and zeros at other positions.

\subsection{VCS and AECS}

To define VCS and AECS, we consider the following problem.
\begin{framed}
\vspace{-1em}
 \begin{align}
\label{prob:unstable}
    \begin{aligned}
        &&& \text{minimize} && h_T(p) \\
        &&& \text{subject to} && p \in X_T\cap \Delta.
    \end{aligned}
\end{align}
\vspace{-1em}
\end{framed}

\noindent
Here, $h_T(p)$ is $f_T(p)$ or $g_T(p)$ on the set $X_T$, which are defined as
\begin{align}
f_T(p) &:= -\log\det W(p, T) \label{def_fT}, \\
    g_T(p) &:= \tr \paren{W(p, T)^{-1}}. \label{def_gT}
\end{align}
Note that $p\in X_T$ means that the virtual system \eqref{eq:lti} is controllable.
Moreover, the constraint $p \in \Delta$ ensures that $p_1,\ldots,p_n$ reflect the relative importance of the corresponding state nodes $x_1,\ldots, x_n$.

For an optimal solution $p$ to Problem \eqref{prob:unstable},
we define the controllability score for state node $x_i$ as follows:
if $h_T(p)=f_T(p)$, then $p_i$ is termed the volumetric controllability score (VCS); 
if $h_T(p)=g_T(p)$, then $p_i$ is the average energy controllability score (AECS). As explained in \cite{sato2022controllability}, each VCS of each state node indicates its importance in enlarging the controllability ellipsoid
\begin{align*}
\mathcal{E}(p,T) := \{y\in\Real^n\mid y^\top W(p,T)^{-1}y \leq 1\}.
\end{align*}
Each AECS of each state node indicates its importance in steering the overall state to a point on the unit sphere. Hereafter, we refer to \eqref{prob:unstable} as the finite-time controllability scoring problem (FTCSP).

Larger VCS and AECS values highlight their significant contribution to the system's controllability.
In fact, since $W_i(T)\neq O$ (because ${\rm tr} (W_i(T))>0$),
for any $p\in X_T$, the $i$-th components of the gradients of $f_T$ and $g_T$ satisfy
\begin{align}
   (\nabla f_T(p))_i &= -{\rm tr}(W(p,T)^{-1}W_i(T))\label{gradf} \\
   &=-{\rm tr} (W(p,T)^{-1/2} W_i(T) W(p,T)^{-1/2}) <0, \nonumber \\
   (\nabla g_T(p))_i &= -{\rm tr}(W(p,T)^{-1}W_i(T)W(p,T)^{-1})<0.  \label{gradg}
\end{align}
Because both gradients are strictly negative,
 increasing $p_i$ enhances the volume of the controllability ellipsoid $\mathcal{E}(p,T)$ and reduces the average energy required for state steering \cite{sato2022controllability}.

Following the approach in \cite[Theorems 1 and 3]{sato2022controllability}, we can prove that problem \eqref{prob:unstable} is a convex optimization problem with an optimal solution, using the compactness of $\Delta$ in ${\bb R}^n$ and the following Hessians for $f_T(p)$ and $g_T(p)$:
  \begin{align}
    &(\nabla^2 f_T(p))_{ij} = \tr (W(p,T)^{-1} W_i(T)W(p,T)^{-1}W_j(T)), \label{Hessian_f}\\
     & (\nabla^2 g_T(p))_{ij} \label{Hessian_g} \\
     &= \tr (W(p,T)^{-1} W_i(T)W(p,T)^{-1}W_j(T)W(p,T)^{-1}) \nonumber\\
      &\,\,\,\, +\tr (W(p,T)^{-1} W_j(T)W(p,T)^{-1}W_i(T)W(p,T)^{-1}). \nonumber
  \end{align}


If the optimal solution to FTCSP \eqref{prob:unstable} is unique, VCS and AECS---which can be calculated using Algorithm \ref{alg:projgrad} proposed in \cite{sato2022controllability}---serve as centrality measures for the network system \eqref{system0}. In this algorithm, $\Pi_\Delta$ at step 2 represents the efficient projection onto the standard simplex $\Delta$ in \eqref{eq:simplex}, as detailed in \cite{condat2016fast}.
In  Algorithm \ref{alg:projgrad}, 
the quantities corresponding to the limit as $T\rightarrow \infty$ of the algorithm proposed in \cite{sato2022controllability} are replaced by
$W_1(T),\ldots, W_n(T)$, $h_T(p)$, and $\nabla h_T(p)$.

Nevertheless,
we can prove the following proposition in the same manner as \cite[Theorem 6]{sato2022controllability}.

\begin{proposition} \label{Prop_convergence}
Suppose that an optimal solution to FTCSP \eqref{prob:unstable} is unique.
If $\{p^{(k)}\}$ is a sequence generated by Algorithm \ref{alg:projgrad} with $\varepsilon=0$, then
    \begin{align*}
        \lim_{k\rightarrow \infty} p^{(k)} = p^*,
    \end{align*}
    where $p^*$ is the optimal solution to  FTCSP \eqref{prob:unstable}, yielding VCS when $h_T=f_T$ and AECS when $h_T=g_T$.
\end{proposition}

\begin{figure}[!t]
\begin{algorithm}[H]
    \caption{A projected gradient method}
    \label{alg:projgrad}
  \textbf{Input:} Controllability Gramians $W_1(T),\ldots, W_n(T)$ in \eqref{eq:finite_time_horizon_gramian}, $p^{(0)} := (1/n,\ldots, 1/n) \in X_T \cap \Delta$, and $\varepsilon\geq 0$.\\
 \textbf{Output:} {VCS or AECS.}
    \begin{algorithmic}[1]
    \FOR{$k=0,1,\ldots$}
    \STATE $p^{(k+1)} := \Pi_{\Delta}(p^{(k)}- \alpha^{(k)} \nabla h_T(p^{(k)}))$, where $\alpha^{(k)}$ is defined by using Algorithm \ref{alg:Armijo}. 
        \IF{$\|p^{(k+1)}-p^{(k)}\|\leq \varepsilon$}
        \RETURN $p^{(k+1)}$.
        \ENDIF
    \ENDFOR
    \end{algorithmic}
\end{algorithm}
\end{figure}

\begin{figure}[!t]
\begin{algorithm}[H]
    \caption{Armijo rule along the projection arc}
    \label{alg:Armijo}
    \textbf{Input:} $\sigma,\,\rho\in (0,1)$ and $\alpha>0$.\\
    \textbf{Output:} {Step size $\alpha^{(k)}$.}
    \begin{algorithmic}[1]
    \STATE $\tilde{p}^{(k)}:=\Pi_{\Delta}(p^{(k)}-\alpha \nabla h_T(p^{(k)}))$.
    \IF{$h_T(\tilde{p}^{(k)})\leq h_T(p^{(k)}) + \sigma \nabla h_T(p^{(k)})^\top (\tilde{p}^{(k)}-p^{(k)})$}
        \RETURN $\alpha^{(k)}:=\alpha$.
    \ELSE
    \STATE $\alpha\leftarrow \rho \alpha$, and go back to step 1.
        \ENDIF
    \end{algorithmic}
\end{algorithm}
\end{figure}

\begin{remark}
    Our VCS and AECS are defined as the optimal solutions to optimization problem \eqref{prob:unstable}. 
Although it is almost trivial--- as shown in \cite{baggio2022energy}---that the objective function values of $f_T(p)$ and $g_T(p)$ differ, the solutions are constrained to lie within $\Delta$ (a normalized range between $0$ and $1$), and thus they are not necessarily distinct.
This normalization guarantees that the differences we observe between VCS and AECS truly reflect the unique aspects of network controllability measured by each metric rather than merely differences in their function value scales.
\end{remark}

\subsection{Existing results}

We summarize the existing results for FTCSP \eqref{prob:unstable}.

The following proposition can be proven in the same way with \cite[Theorem 2]{sato2022controllability} for $T\rightarrow \infty$, although our analysis is conducted for finite $T>0$.

\begin{proposition} \label{Prop_stable}
    If system \eqref{system0} is asymptotically stable, for all $T>0$,
   FTCSP \eqref{prob:unstable} admits a unique optimal solution, yielding VCS when $h_T=f_T$ and AECS when $h_T=g_T$.
\end{proposition}

The following has been shown in \cite[Theorem 4]{sato2022controllability}.

\begin{proposition}\label{Prop_unstable_realeigen}
     Assume that $A$ and $-A$ of system \eqref{system0} do not have a common eigenvalue.
    Then, for all $T>0$,
   FTCSP \eqref{prob:unstable} admits a unique optimal solution, yielding VCS when $h_T=f_T$ and AECS when $h_T=g_T$.
\end{proposition}

The following has been shown in \cite[Theorem 5]{sato2022controllability}.

\begin{proposition} \label{Prop_Laplacian}
    Let $L=(\ell_{ij})$ be a graph Laplacian matrix corresponding to an undirected connected graph.
    That is, $L$ is symmetric, and the non-diagonal element $\ell_{ij}$ $(i \neq j)$ is defined as $\ell_{ij} = -c_{ij}$ if there is an edge from state node $j$ to state node $i$ with weight $c_{ij} > 0$, and $\ell_{ij} = 0$ if no such edge exists. The diagonal element $\ell_{ii}$ is given by $\ell_{ii} = -\sum_{j \neq i} \ell_{ij}$.
    Then, for $A=-L$ and all $T>0$,
   FTCSP \eqref{prob:unstable} admits a unique optimal solution, yielding VCS when $h_T=f_T$ and AECS when $h_T=g_T$.
\end{proposition}

The interpretations of Propositions \ref{Prop_stable}, \ref{Prop_unstable_realeigen}, \ref{Prop_Laplacian} are as follows:
Proposition \ref{Prop_stable} states that VCS and AECS are uniquely defined for any asymptotically stable system \eqref{system0} and for all $T>0$.
Proposition \ref{Prop_unstable_realeigen} suggests that VCS and AECS remain uniquely defined for any system whose eigenvalues lie off the imaginary axis, again for all $T>0$.
Proposition \ref{Prop_Laplacian} indicates that VCS and AECS are uniquely defined for any symmetric Laplacian dynamics \eqref{system0} with $A=-L$ corresponding to an undirected graph, for all $T>0$.

\subsection{Limitations of existing results and objectives of this paper} \label{Subsec_limitation}

While Propositions \ref{Prop_convergence}, \ref{Prop_stable}, \ref{Prop_unstable_realeigen}, and \ref{Prop_Laplacian} provide valuable insights, they are insufficient in the following respects:
\begin{enumerate}[(i)]
\item It remains unclear whether VCS and AECS are uniquely defined for all linear systems of the form \eqref{system0}, especially in the case of non-symmetric Laplacian dynamics arising from directed graphs.
\item The propositions do not specify the class of systems for which VCS and AECS differ.

\item The investigation into the effect of the terminal time
$T$ for VCS and AECS is lacking.

\item Convergence analysis of Algorithm \ref{alg:projgrad} is not sufficient.

\item The effectiveness of VCS and AECS in real-world network systems is not clear.
\end{enumerate}

This paper aims to resolve and validate the aforementioned points (i)--(v) in the following sections.

\section{Analyses}

In this section, we address points (i)--(iv) as outlined in Section \ref{Subsec_limitation}, while point (v) is discussed in Section \ref{Sec_numerical}.

\subsection{Uniqueness of VCS and AECS} \label{sec:uniqueness}

We prove the uniqueness of the VCS and AECS for any linear systems of the form \eqref{system0} for almost all $T>0$.
That is, we resolve (i), as mentioned in Section \ref{Subsec_limitation}.

\begin{theorem} \label{thm:unique_anyA}
For all $A\in \Real ^{n\times n}$ and almost all $T>0$,
   FTCSP \eqref{prob:unstable} admits a unique optimal solution, yielding VCS when $h_T=f_T$ and AECS when $h_T=g_T$.
\end{theorem}

To prove Theorem \ref{thm:unique_anyA},
we prepare
the following lemma, which can be proven in the same way with \cite[Lemma 2 and Theorem 1]{sato2022controllability} for the case $T\rightarrow \infty$, even though our analysis is conducted for finite $T>0$.

\begin{lemma}\label{lem:sufficient condition of uniqueness}
   Let $T>0$ be arbitrary. If $W(x,T)=O$ implies $x=0$, then the solution to \eqref{prob:unstable} is unique.
\end{lemma}

Moreover, we recall the following lemma, which has been proven in \cite{mityagin2015zero}.

\begin{lemma} \label{lem:measure zero}
    Let $\varphi$ be a real analytic function on $\Real$. If $\varphi(x) \not\equiv 0$, then the Lebesgue measure of the zero set $\{x \in \Real \mid \varphi(x) = 0\}$ is $0$.
\end{lemma}

Lemma \ref{lem:measure zero} means that an analytic function that is not identically zero is not zero at almost all $x \in \Real$.

\textit{Proof of Theorem \ref{thm:unique_anyA}:}
From Lemma \ref{lem:sufficient condition of uniqueness}, it is sufficient to show that for almost all $T>0$ and all $x=(x_i) \in \Real ^n$, $W(x, T) = O$ yields $x=0$.
Note that $W_i(T)$ is defined as
    $W_i(T) = \int _0^T P(t) e_i e_i^\top P(t)^\top \D t$
with $P(t) := \E^{At}$.
For $i = 1,2,\ldots,n$, $(i,i)$-th component of $W(x,T)$ is obtained as
\begin{align}
    \left(W(x,T)\right)_{ii}
    &= e_i^\top W(x,T)e_i\\
    &= \sum_{j=1}^n x_j \cdot \int_0^T \left( e_i^\top P(t) e_j \right)^2\D t\\
    &= \sum_{j=1}^n x_j \cdot \int_0^T  P_{ij}(t)^2\D t, \label{eq:ii-th component}
\end{align}
where $P_{ij}(t)$ denotes the $(i,j)$ element of $P(t)$.
Eq. \eqref{eq:ii-th component} implies that $W(x,T)=O$ yields
\begin{align}
R(T)x=0, \label{eq:R(T)x = 0}
\end{align}
where
\begin{align}
    R(T) := \int_0^T \begin{bmatrix}
    P_{11}(t)^2 & P_{12}(t)^2 &\cdots& P_{1n}(t)^2 \\
    P_{21}(t)^2 & P_{22}(t)^2 &\cdots & P_{2n}(t)^2\\
    \vdots & \vdots & \ddots & \vdots\\
    P_{n1}(t)^2 & P_{n2}(t)^2 & \cdots & P_{nn}(t)^2
    \end{bmatrix}
    \D t.
\end{align}
If $\det R(T) \neq 0$, \eqref{eq:R(T)x = 0} implies $x=0$.
Thus, in what follows, we show that $\det R(T) \neq 0$ for almost all $T>0$.

To this end, let $\varphi(T) := \det R(T)$.
Note that $\varphi(T)$ is an analytic function on $\Real$, because each element of $P(t) = \E^{At}$ is represented by the finite sum, difference or product of $t$ or exponential functions of $t$ or trigonometric functions of $t$.
Thus, if $\dfrac{\D^n \varphi}{\D T^n}(0) \neq 0$,
$\varphi(T)$ is not a zero function by
considering the Taylor expansion of $\varphi(T)$ at $T=0$.
Therefore, Lemma \ref{lem:measure zero} implies that Lebesgue measure of the zero set of $\varphi(T)$ is 0, that is, $\det R(T) \neq 0$ holds for almost all $T>0$.

To show that $\dfrac{\D^n \varphi}{\D T^n}(0) \neq 0$, we use the following properties of $R(T)$:
\begin{align}
    R(0) &= O,\label{eq:R(0) = O}\\
    \dfrac{\D R}{\D T} (0) &= \begin{bmatrix}
    P_{11}(0)^2 & P_{12}(0)^2 &\cdots& P_{1n}(0)^2 \\
    P_{21}(0)^2 & P_{22}(0)^2 & \cdots & P_{2n}(0)^2\\
    \vdots & \vdots & \ddots & \vdots\\
    P_{n1}(0)^2 & P_{n2}(0)^2 & \cdots & P_{nn}(0)^2
    \end{bmatrix}\\
    &=I,  \label{eq: D R(0) = I}
\end{align}
where \eqref{eq: D R(0) = I} follows from $P(0) = \E^{A\cdot 0} = I$.
By the definition of determinant, $\varphi(T)$ is represented as
    \begin{align}
        \varphi(T) = \sum_{\sigma \in S_n} \mathrm{sgn}(\sigma) \prod_{k=1}^n R_{k\sigma(k)}(T),
    \end{align}
    where $R_{ij}(t)$ denotes the $(i,j)$ element of $R(t)$.
    Because $\prod_{k=1}^n R_{k\sigma(k)}(T)$ is the product of $n$ elements of $R(T)$,
     $\dfrac{\D^n \varphi}{\D T^n}(0)$ leaves only terms where each $n$ element is differentiated exactly once, and the other terms vanish from \eqref{eq:R(0) = O}.
Notably, there are $n!$ ways for differentiating $n$ terms exactly once.
Thus, $\dfrac{\D^n \varphi}{\D T^n}(0)$ can be calculated as
    \begin{align}
        \dfrac{\D^n \varphi}{\D T^n}(0) 
        &= \sum_{\sigma \in S_n} \mathrm{sgn}(\sigma) n!\prod_{k=1}^n \dfrac{\D R_{k\sigma(k)}}{\D T} (0)\\
        &= \sum_{\sigma \in S_n} \mathrm{sgn}(\sigma) n!\prod_{k=1}^n I _{k\sigma(k)} \label{eq:D^n varphi(0)}\\
        &= n!\cdot \det I
        = n!\neq 0,    \end{align}
where \eqref{eq:D^n varphi(0)} follows from $\dfrac{\D R}{\D T}(0) = I$.
Thus, as mentioned earlier, $\det R(T) \neq 0$ holds for almost all $T>0$. \qed

By Theorem \ref{thm:unique_anyA}, VCS and AECS can be used as centrality measures for network systems. The uniqueness of these measures is crucial for ensuring interpretability, comparability, and reproducibility across different researchers.

Note that we cannot replace ``almost all $T$" in Theorem \ref{thm:unique_anyA} with ``all $T$",
because there is an example where a solution to FTCSP \eqref{prob:unstable} is not unique, as shown in \cite[Section IV]{sato2022controllability}. 
The intuition behind the non-uniqueness of the solution for certain values of $T$ is linked to the spectral properties of the matrix $A$. Specifically, when 
$A$ has eigenvalues on the imaginary axis, the system exhibits marginally stable or oscillatory behavior. For certain time horizons
$T$, this oscillatory behavior can cause the controllability Gramian to become insensitive with respect to variations in the input weight distribution
$p$. In other words, different distributions of $p$ can result in the same value of the objective function for VCS or AECS. This flatness in the optimization landscape leads to multiple solutions that all yield the same optimal performance, thereby explaining the observed non-uniqueness.

\subsection{VCS and AECS in special cases}
\label{sec:special}
We show 
a class of system \eqref{system0} that yields differences between VCS and AECS.
That is, we resolve (ii), as mentioned in Section \ref{Subsec_limitation}.
Moreover, we clarify a class of system \eqref{system0} that yields the same results for VCS and AECS.

To this end, we prepare the following.

\begin{lemma} \label{Lem_sufficient}
     Let $p^*\in X_T\cap \Delta$.
    If $\nabla h_T(p^*) = k{\bf 1}$ for some $k\in {\bb R}$, $p^*$ is an optimal solution to FTCSP \eqref{prob:unstable}, yielding VCS when $h_T=f_T$ and AECS when $h_T=g_T$.
\end{lemma}
\begin{proof}
According to \cite[Proposition 3.1.1]{bertsekas2016nonlinear},
$p^*$ is an optimal solution to convex optimization problem \eqref{prob:unstable} if and only if
\begin{align}
\nabla h_T(p^*)^\top (p - p^*) \geq 0 \label{daiji}
\end{align}
 for all $p\in X_T\cap \Delta$.
If $\nabla h_T(p^*) = k{\bf 1}$ for some $k\in {\bb R}$,
\begin{align*}
    \nabla h_T(p^*)^\top (p - p^*) &= k{\bf 1}^\top (p-p^*)\\
    &=k\sum_{i=1}^n (p_i-p_i^*) = k(1-1) =0
\end{align*}
 for all $p\in X_T\cap \Delta$.
Thus, \eqref{daiji} holds. \qed
\end{proof}

Using Lemma \ref{Lem_sufficient}, we obtain the following.

\begin{theorem} \label{Thm_symmetric}
    If $A$ in \eqref{eq:lti} is symmetric, i.e., $A^\top = A$, then ${\bf 1}/n$ is an optimal solution to FTCSP \eqref{prob:unstable} with $h_T(p)=f_T(p)$ for all $T>0$.
\end{theorem}
\begin{proof}
We show that $\nabla f_T({\bf 1}/n)= -n{\bf 1}$ with $h_T(p)=f_T(p)$ for all $T>0$, because in this case, Lemma \ref{Lem_sufficient} guarantees that ${\bf 1}/n$ is an optimal solution to FTCSP \eqref{prob:unstable} with $h_T(p)=f_T(p)$ for all $T>0$.
To this end, we note that there exist an orthogonal matrix $U$ and a real diagonal matrix $D$ such that $A=UDU^\top$ due to the symmetry of $A$.

Because $\nabla f_T({\bf 1}/n)$ is given by \eqref{gradf} with $p={\bf 1}/n$,
we calculate
  \begin{align*}
    W({\bf 1}/n,T) &= \int_0^T \E^{At} \cdot \dfrac{1}{n}I\cdot \E^{A^\top t}\D t\\
    &=\dfrac{1}{n}U\int_0^T\E^{Dt}U^\top U\E^{Dt}\D t U^\top\\
    &=\dfrac{1}{n}U\left(\int_0^T\E^{2Dt}\D t\right) U^\top
    = \dfrac{1}{n}UFU^\top,
  \end{align*}
where $F=\int_0^T \E^{2Dt}\D t$ is an invertible diagonal matrix.
Thus,
  \begin{align*}
   & (\nabla f_T({\bf 1}/n))_i \\
    &= -n\cdot \tr \left(UF^{-1}U^\top W_i(T)\right)\\
    &= -n\cdot \tr \left(UF^{-1}U^\top \int_0^T U\e^{Dt}U^\top e_i e_i^\top U\e^{Dt}U^\top \diff t\right)\\
    &= -n\cdot \int_0^T\tr \left(UF^{-1}\e^{Dt}U^\top e_i e_i^\top U\e^{Dt}U^\top \right)\diff t\\
    &= -n\cdot \int_0^T\tr \left(U\e^{Dt}U^\top UF^{-1}\e^{Dt}U^\top e_i e_i^\top  \right)\diff t\\
    &= -n\cdot \tr \left(U\int_0^T\e^{Dt}F^{-1}\e^{Dt}\diff tU^\top e_i e_i^\top  \right)\\
    &= -n\cdot \tr \left(UF^{-1}\int_0^T\e^{2Dt}\diff tU^\top e_i e_i^\top  \right)\\
    &= -n\cdot \tr \left(UF^{-1}FU^\top e_i e_i^\top  \right)\\
    &= -n\cdot \tr \left(e_i e_i^\top  \right)= -n.
  \end{align*}
  This completes the proof. \qed
\end{proof}

According to Theorem \ref{Thm_symmetric}, if $A$ is close to being symmetric, we expect each node's VCS to lie closer to the uniform value ${\bf 1}/n$ than its corresponding AECS value.

Proposition \ref{Prop_Laplacian} and Theorem \ref{Thm_symmetric} imply the following.

\begin{corollary} \label{Cor_Laplacian}
    Let $L$ be a graph Laplacian matrix corresponding to an undirected connected graph with $n$ nodes. Then, for $A=-L$,
    ${\bf 1}/n$ is the unique solution to
FTCSP \eqref{prob:unstable} with $h_T(p)=f_T(p)$ for all $T>0$.
\end{corollary}

Corollary \ref{Cor_Laplacian} states that for symmetric Laplacian dynamics on an undirected graph with $n$ nodes, the VCS of each state node is uniformly $1/n$.

\begin{table}[t]
\centering
\caption{VCS and AECS when $A$ is symmetric or skew-symmetric.}
\label{table0}
\begin{tabular}{|c|c|c|}
\hline
$A$ & symmetric & skew-symmetric \\ \hline
VCS & $={\bf 1}/n$ & $={\bf 1}/n$ \\ 
AECS & $\neq {\bf 1}/n$ & $={\bf 1}/n$  \\ \hline
\end{tabular}
\end{table}

\begin{figure}[t]
    \centering
    \includegraphics[width=3cm]{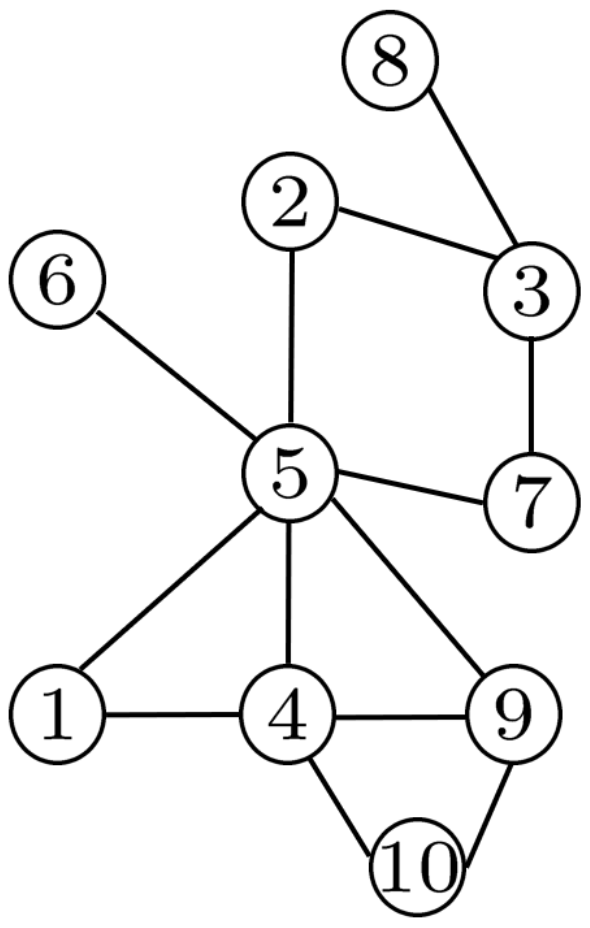}
    \caption{Connected undirected Network. State node \textcircled{\scriptsize 5}, which has degree 6, is expected to exhibit greater controllability than state nodes \textcircled{\scriptsize 6} and \textcircled{\scriptsize 8}, which have degree 1.}
    \label{fig:undirected}
\end{figure}

\begin{table}[t]
\caption{VCS and AECS for the undirected network system illustrated in Fig. \ref{fig:undirected}.}
\label{table1}
\centering
\begin{tabular}{c||c|c|c|c|c}
Node $i$ & VCS & AECS \\ \hline
1 &   0.1000& 0.0803 \\
2 &     0.1000&    0.1096 \\
3 &     0.1000&    0.1185\\
4 &     0.1000&    0.1308 \\
5 &     0.1000&    0.1613\\
6 &     0.1000&    0.0813\\
7 &     0.1000&    0.0831 \\
8 &     0.1000&    0.0466\\
9 &     0.1000&    0.1088 \\
10 &    0.1000&     0.0799
\end{tabular}
\end{table}

Moreover, we obtain the following.

\begin{theorem} \label{Thm_skew}
    If $A$ in \eqref{eq:lti} is skew-symmetric, i.e., $A^\top = -A$, then ${\bf 1}/n$ is an optimal solution to FTCSP \eqref{prob:unstable} for all $T>0$.
\end{theorem}
\begin{proof}
Because $\nabla f_T({\bf 1}/n)$ and $\nabla g_T({\bf 1}/n)$ are given by \eqref{gradf} and \eqref{gradg} with $p={\bf 1}/n$, respectively,
we calculate
\begin{align*}
    W({\bf 1}/n,T)
    &= \dfrac{1}{n}\int_0^T \e^{At}\e^{A^\top t}\diff t\\
    &= \dfrac{1}{n}\int_0^T \e^{At}\e^{-A t}\diff t= \dfrac{T}{n}I.
  \end{align*}
Thus, $\nabla f_T({\bf 1}/n)$ is given by
\begin{align*}
    (\nabla f_T({\bf 1}/n))_i
    &=-\dfrac{n}{T}\cdot \tr\left(\int_0^T \e^{At}e_ie_i^\top \e^{A^\top t}\diff t\right)\\
    &=-\dfrac{n}{T}\cdot \tr\left(\int_0^T \e^{A^\top t}\e^{At} \diff t e_ie_i^\top \right)\\
    &=-\dfrac{n}{T}\cdot \tr\left(T\cdot I e_ie_i^\top \right)=-n.
    \end{align*}
    Moreover, $\nabla g_T({\bf 1}/n)$ is given by
    \begin{align*}
(\nabla g_T({\bf 1}/n))_i          
    =-\dfrac{n^2}{T^2}\cdot \tr(W_i(T))=-\dfrac{n^2}{T}.
  \end{align*}
  Here, the second equality follows from
       $\tr(W_i(T)) = \int_0^T \tr (e^{At}e_ie_i^\top e^{A^\top t}) \D t = \int_0^T \tr (e_i e_i^\top) \D t=T$.
  This completes the proof. \qed
\end{proof}

Theorem \ref{Thm_skew} means that
for systems of the form \eqref{system0} with skew-symmetric $A$, 
both VCS and AECS may yield a uniform value.
This uniformity indicates that, in such an idealized setting, the two measures do not offer additional distinguishing insights into the system’s behavior. It is important to emphasize that a strictly skew-symmetric $A$ corresponds to an energy-preserving system—a scenario that is rarely encountered in practice due to the presence of damping and external disturbances. Therefore, the coincidence of VCS and AECS under strict skew-symmetry represents a theoretical special case rather than a realistic scenario. In most practical applications, where $A$ deviates from being strictly skew-symmetric, VCS and AECS are expected to provide distinct and informative evaluations that better reflect the system's true dynamics.

Although Theorems \ref{Thm_symmetric} and \ref{Thm_skew} do not guarantee the uniqueness of VCS and AECS, combining them with Theorem \ref{thm:unique_anyA} ensures the uniqueness for almost all $T>0$,
as shown in Table \ref{table0}.

 Note that even if
$A$ is symmetric, the AECS for each state node is not uniformly equal to $1/n$. 
In fact,  consider the
Laplacian dynamics of the form \eqref{system0} with $A=-L$
on the undirected graph in Fig. \ref{fig:undirected}, 
where $L$ is a graph Laplacian with uniform edge weights.
Then,
Algorithm \ref{alg:projgrad} for $T=1000$ produced VCS and AECS results as shown in Table \ref{table1}. Each state node had a VCS value of $1/10$, in accordance with Corollary \ref{Cor_Laplacian}. In contrast, AECS analysis differentiates nodes based on their degrees. Specifically, nodes with higher degrees have larger AECS values, highlighting their pivotal roles in controlling the network's dynamics. Conversely, node \textcircled{\scriptsize 8}, with a lower degree, has a smaller AECS, indicating the correlation between a node's degree and its network controllability.

Therefore, there exists a class of systems for which VCS and AECS yield different results. We further explore this fact in Section \ref{Sec_numerical}.

\subsection{Effect of the terminal time
$T$} \label{Sec_effect_T}

We investigate the effect of the terminal time
$T$ on the optimal solution to FTCSP \eqref{prob:unstable}.
That is, we partially resolve (iii), as mentioned in Section \ref{Subsec_limitation}. 
To this end, we define 
    $Z_i(t):= \exp(At) e_ie_i^\top \exp(A^\top t)$.
Then, the finite time controllability Gramian $W_i(T)$ in \eqref{eq:finite_time_horizon_gramian} can be expressed as
\begin{align}
    W_i(T) = \int_0^T Z_i(t) \D t. \label{W_Z}
\end{align}

\subsubsection{When $T$ is close to zero}

The following lemma is important for analyzing the effect of the terminal time $T$.

\begin{lemma} \label{Lemma_T_small}
For the finite time controllability Gramian $W(p,T)$ in \eqref{Def_Wc}, there exists $\delta >0$ such that for all $T\in [0,\delta]$,
\begin{align}
    W(p,T) = T{\rm diag}(p_1,\ldots, p_n) + O(T^2). \label{W_another}
\end{align}
\end{lemma}
\begin{proof}
Since $Z_i$ is analytic, the Taylor's theorem implies 
\begin{align*}
    Z_i(t) = Z_i(0) +\frac{\D Z_i}{\D t}(0) t + \int_0^t \frac{\D^2 Z_i}{\D t^2}(\tau) (t-\tau) \D \tau.
\end{align*}
Since \eqref{W_Z} holds,
\begin{align}
    W_i(T) =& Z_i(0)T + \frac{1}{2}\frac{\D Z_i}{\D t}(0)T^2 \nonumber\\ 
    &+ \int_0^T\int_0^t \frac{\D^2 Z_i}{\D t^2}(\tau) (t-\tau) \D \tau \D t. \label{eq_Wi_int}
\end{align}
Because the integrand $\frac{\D^2 Z_i}{\D t^2}(\tau) (t-\tau)$ of the third term is integrable, the Fubini's theorem guarantees that we can interchange the order of integration. Thus,
\begin{align}
  & \int_0^T\int_0^t \frac{\D^2 Z_i}{\D t^2}(\tau) (t-\tau) \D \tau \D t \nonumber\\
  =& \int_0^T \frac{\D^2 Z_i}{\D t^2}(\tau) \left( \int_{\tau}^T (t-\tau) \D t\right) \D \tau \nonumber\\
  =& \frac{1}{2}\int_0^T \frac{\D^2 Z_i}{\D t^2}(\tau)(T-\tau)^2 \D \tau. \label{eq_double_int}
\end{align}
By substituting \eqref{eq_double_int} into \eqref{eq_Wi_int}, we obtain
\begin{align}
\|W_i(T) - Z_i(0)T\| \leq& \frac{1}{2}\left\|\frac{\D Z_i}{\D t}(0)\right\|T^2 \label{W_i(T)_expression0}\\
&+ \frac{1}{2}\int_0^T\left\|\frac{\D^2 Z_i}{\D t^2}(\tau)\right\|(T-\tau)^2 \D \tau. \nonumber
\end{align}
Since $\frac{\D^2 Z_i}{\D t^2}(t)$ is continuous in $t$,
for each $i\in \{1,\ldots, n\}$, there exists $\delta_i>0$ such that $\|\frac{\D^2 Z_i}{\D t^2}(t)\| \leq 2\|\frac{\D^2 Z_i}{\D t^2}(0)\|$ for all $t\in [0,\delta_i]$, where $\|\cdot\|$ denotes any matrix norm.
Define
\begin{align*}
    \delta:= \min_{i\in\{1,\ldots, n\}} \delta_i,\quad M:=\max_{i\in\{1,\ldots, n\}} \left\|\frac{\D^2 Z_i}{\D t^2}(0) \right\|.
\end{align*}
Then, \eqref{W_i(T)_expression0} implies that for any $T\in [0,\delta]$,
\begin{align}
\|W_i(T) - Z_i(0)T\| &\leq \frac{1}{2}\left\|\frac{\D Z_i}{\D t}(0)\right\|T^2 + \frac{1}{3}M T^3 \nonumber \\ 
&\leq C T^2, \label{W_i(T)_expression}
\end{align}
where $C:=\frac{1}{2}\left\|\frac{\D Z_i}{\D t}(0)\right\| + \frac{1}{3}M\delta$.
Because $Z_i(0)=e_ie_i^\top$ and
    $\|W(p,t) -T{\rm diag}(p_1,\ldots, p_n)\|\leq \sum_{i=1}^n |p_i| \|W_i(T)-Z_i(0)T\|$,
\eqref{W_i(T)_expression} yields \eqref{W_another}. \qed
\end{proof}

Based on Lemma \ref{Lemma_T_small},
if $T$ is sufficiently close to zero,
we can replace
$f_T(p)$ in
\eqref{def_fT} and $g_T(p)$ in \eqref{def_gT} with
\begin{align}
    \widetilde{f}_T(p) &:= -\log \det \widetilde{W}(p,T)=-n\log T -\sum_{i=1}^n \log p_i, \\
    \widetilde{g}_T(p) &:= {\rm tr}(\widetilde{W}^{-1}(p,T))= \frac{1}{T} \left( \sum_{i=1}^n\frac{1}{p_i}\right),
\end{align}
respectively, where
    $\widetilde{W}(p,T) := T{\rm diag}(p_1,\ldots, p_n)$.
This substitution is justified because Lemma \ref{Lemma_T_small} implies that
$W(p,T)$ can be approximated by $\widetilde{W}(p,T)$ when $T$ is sufficiently small.

Thus, if $T$ is sufficiently close to zero, FTCSP \eqref{prob:unstable} can be approximated to the following problem, where $\widetilde{h}_T(p)$ is $\widetilde{f}_T(p)$ or
$\widetilde{g}_T(p)$.

\begin{framed}
\vspace{-1em}
 \begin{align}
\label{prob:unstable2}
    \begin{aligned}
        &&& \text{minimize} && \widetilde{h}_T(p) \\
        &&& \text{subject to} && p \in \Delta_{\rm int}.
    \end{aligned}
\end{align}
\vspace{-1em}
\end{framed}

Note that the original constraint $X_T\cap \Delta$ in FTCSP \eqref{prob:unstable} was replaced with 
$\Delta_{\rm int}$, which denotes the interior of $\Delta$. Specifically, $\Delta_{\rm int}$ is defined as
\begin{align}
    \Delta_{\rm int} := \left\{ p = (p_i) \in \Real^n \left|
    \begin{array}{l}
         \sum_{i=1}^n p_i = 1,    \\
         0 < p_i \quad (i=1,\dots,n)  
    \end{array}
    \right. \right\}. \label{eq:simplex_int}
\end{align}
This substitution is necessary because $\widetilde{f}_T(p)$ and $\widetilde{g}_T(p)$ are not defined unless each component of the vector $p$ is positive.
Moreover, the new constraint $\Delta_{\rm int}$ of problem \eqref{prob:unstable2} is contained within the original constraint $X_T\cap \Delta$ of FTCSP \eqref{prob:unstable}, because
for any $p\in \Delta_{\rm int}$, system \eqref{eq:lti} is controllable, implying that $\Delta_{\rm int}\subset X_T$.

    Since $\Delta_{\rm int}$ is not closed in ${\bb R}^n$,
    determining whether an optimal solution exists for 
 problem \eqref{prob:unstable2} is a nontrivial issue.
 However, we can prove that an optimal solution does exist and is moreover unique, as shown below.

\begin{theorem} \label{Thm_T_small}
    Suppose that $T>0$ is given.
    Then, problem \eqref{prob:unstable2} has a unique optimal solution, which is given by
    $p={\bf 1}/n$ for  $\widetilde{h}_T(p)=\widetilde{f}_T(p)$ and $\widetilde{h}_T(p)=\widetilde{g}_T(p)$.
\end{theorem}
\begin{proof}
We begin by computing the Hessians of $\widetilde{f}_T(p)$ and $\widetilde{g}_T(p)$ on $\Delta_{\rm int}$:
\begin{align*}
    \nabla^2 \widetilde{f}_T(p) &= {\rm diag}(1/p_1^2,\ldots, 1/p_n^2), \\
        \nabla^2 \widetilde{g}_T(p) &= \frac{2}{T}{\rm diag}(1/p_1^3,\ldots, 1/p_n^3).
\end{align*}
Since
both $\nabla^2 \widetilde{f}_T(p)$ and $\nabla^2 \widetilde{g}_T(p)$ are positive definite on $\Delta_{\rm int}$,
it follows from 
\cite[Proposition B.4]{bertsekas2016nonlinear}
that $\widetilde{f}_T(p)$ and $\widetilde{g}_T(p)$ are strictly convex on $\Delta_{\rm int}$.

Next, we introduce the Lagrangians associated with the equality constraint $\sum_{i=1}^n p_i = 1$. 
For $\widetilde{f}_T(p)$, define
\begin{align*}
    L_{f}(p,\lambda) &:= \widetilde{f}_T(p) + \lambda \left( \sum_{i=1}^n p_i-1 \right),
    \end{align*}
    and for $\widetilde{g}_T(p)$, define
    \begin{align*}
    L_{g}(p,\lambda) &:= \widetilde{g}_T(p) + \lambda \left( \sum_{i=1}^n p_i-1 \right),
\end{align*}
where $\lambda\in\mathbb{R}$ is the Lagrange multiplier for the constraint.

Taking the partial derivatives,
we obtain
\begin{align}
    \frac{\partial L_f}{\partial p_i}(p,\lambda) = -\frac{1}{p_i} + \lambda, \,\,
    \frac{\partial L_g}{\partial p_i}(p,\lambda) = -\frac{1}{Tp_i^2} + \lambda
\end{align}
 for any $p\in \Delta_{\rm int}$ and $\lambda \in {\bb R}$.
Thus, setting
$\frac{\partial L_f}{\partial p_i}(p,\lambda)=0$ 
yields $\lambda = 1/p_i$ for each $i=1,\ldots, n$, implying that the optimal solution satisfies $p_1=\cdots = p_n$.
Combining this with the constraint $\sum_{i=1}^np_i = 1$,
we obtain  $p_1=\cdots =p_n = 1/n$, which is an optimal solution to problem \eqref{prob:unstable2} with $\widetilde{h}_T(p) = \widetilde{f}_T(p)$.
Similarly, setting
$\frac{\partial L_g}{\partial p_i}(p,\lambda)=0$ yields
$\lambda = 1/(Tp^2_i)$ for each $i=1,\ldots, n$, implying that the optimal solution satisfies $p_1=\cdots = p_n$.
Thus, $p_1=\cdots =p_n = 1/n$, which is an optimal solution to problem \eqref{prob:unstable2} with $\widetilde{h}_T(p) = \widetilde{g}_T(p)$.

Since $\widetilde{f}_T(p)$ and $\widetilde{g}_T(p)$ are strictly convex on the convex set $\Delta_{\rm int}$, it follows from \cite[Proposition 1.1.2]{bertsekas2016nonlinear} that $p=\mathbf{1}/n$ is the unique optimal solution for both cases.
\qed
\end{proof}

Theorem \ref{Thm_T_small} guarantees that if $T$ is sufficiently small, the optimal solution---whether VCS or AECS---to FTCSP \eqref{prob:unstable} is close to ${\bf 1}/{n}$.
In other words, when $T$ is very small,  both VCS and AECS essentially become uniform, thereby failing to provide any meaningful information about the relative importance of the nodes.

\subsubsection{When $T$ is sufficiently large}

Since Lemma \ref{Lemma_T_small} holds only for sufficiently small $T$, it cannot be applied when $T$ is large. As a result, the analysis for large $T$ is, in general, challenging.

Therefore, we assume that
 $A$ is stable.
Here, we call $A$ stable if all eigenvalues of $A$ have negative real parts, and we say system \eqref{system0} to be asymptotically stable if $A$ is stable.

Then, we can define
\begin{align}
\label{eq:infinite_time_horizon_gramian}
    W_i^{\infty} := \int_0^\infty \exp(At) e_ie_i^\top \exp(A^\top t) \ \D t.
\end{align}
In this case, for a given $T>0$,
    $W_i^{\infty} = W_i(T) + \widehat{W}_i(T)$,
where
\begin{align}
    \widehat{W}_i(T):=\int_T^\infty \exp(At) e_ie_i^\top \exp(A^\top t) \ \D t. \label{Def_hat_W}
\end{align}

The following theorem shows that if $A$ is stable and $T$ is sufficiently large, then $W_i^{\infty}$ is close to $W_i(T)$.

\begin{theorem} \label{Thm_T_stable}
    Suppose that $A$ in \eqref{system0} is stable and diagonalizable.
    That is, there exists an invertible matrix $P\in {\bb C}^{n\times n}$ such that
    \begin{align}
        A = P^{-1}{\rm diag}(\lambda_1,\ldots, \lambda_n)P, \label{A_diag}
    \end{align}
    where $\lambda_i\in {\bb C}$ denotes the eigenvalue of $A$, and ${\rm Re}(\lambda_i)<0$ holds 
    for $i=1,\ldots, n$.
    For any positive $\varepsilon$ satisfying $\varepsilon <c$, if
    \begin{align}
    T>\frac{\log (\varepsilon/c)}{2\alpha}, \label{Ineq_T}
    \end{align} 
    then
    \begin{align}
        \|\widehat{W}_i(T)\| < \varepsilon, \label{Ineq_small}
    \end{align}
     where the above $\|\cdot\|$ denotes the operator norm, and
    \begin{align*}
    \alpha := \max\{{\rm Re}(\lambda_1),\ldots, {\rm Re}(\lambda_n)\}, \,
        c := -\frac{\|P\|^2\ \|P^{-1}\|^2}{2\alpha}.
    \end{align*}
\end{theorem}
\begin{proof}
    It follows from \eqref{Def_hat_W} and \eqref{A_diag} that
    \begin{align*}
        \|\widehat{W}_i(T)\| \leq \|P\|^2 \|P^{-1}\|^2 \int_T^\infty e^{2\alpha t} \D t
        = c e^{2\alpha T},
    \end{align*}
    where we have used $\alpha<0$. Thus, if \eqref{Ineq_T} holds, \eqref{Ineq_small} is satisfied. \qed
\end{proof}

Theorem \ref{Thm_T_stable} is important for the following reason: \(W_i^{\infty}\) can be computed more efficiently than \(W_i(T)\), because it can be obtained via the solution of a Lyapunov equation. Theorem \ref{Thm_T_stable} applies in the case where \(A\) is stable, and it implies that when \(T\) is sufficiently large, both AECS and VCS become independent of \(T\). However, if \(A\) is not stable, AECS and VCS may still depend on \(T\). 
Nevertheless, the following example suggests that when \(A\) is associated with a graph Laplacian that has a zero eigenvalue, as \(T\) increases, AECS and VCS eventually become independent of \(T\).

\begin{table*}[t]
\centering
\caption{Evaluation of VCS and AECS for various $T$.}
\label{tab:combined}
\begin{minipage}[b]{0.48\textwidth}
\centering
\textbf{VCS}\\[0.5ex]
\begin{tabular}{c c c c c}
\toprule
\textbf{Node} & \textbf{$T=0.01$} & \textbf{$T=1$} & \textbf{$T=1000$} & \textbf{$T=10000$} \\
\midrule
1  & 0.1      & 0.099677 & 0.073347 & 0.073327 \\
2  & 0.1      & 0.1      & 0.10112  & 0.10108  \\
3  & 0.1      & 0.1      & 0.10876  & 0.10874  \\
4  & 0.1      & 0.099997 & 0.086378 & 0.086362 \\
5  & 0.1      & 0.099674 & 0.045557 & 0.044985 \\
6  & 0.1      & 0.09935  & 0.060743 & 0.060707 \\
7  & 0.1      & 0.10131  & 0.24929  & 0.24952  \\
8  & 0.1      & 0.09967  & 0.042309 & 0.042214 \\
9  & 0.1      & 0.10033  & 0.16614  & 0.16674  \\
10 & 0.1      & 0.099995 & 0.066358 & 0.066317 \\
\bottomrule
\end{tabular}
\end{minipage}%
\hfill
\begin{minipage}[b]{0.48\textwidth}
\centering
\textbf{AECS}\\[0.5ex]
\begin{tabular}{c c c c c}
\toprule
\textbf{Node} & \textbf{$T=0.01$} & \textbf{$T=1$} & \textbf{$T=1000$} & \textbf{$T=10000$} \\
\midrule
1  & 0.1      & 0.10927  & 0.17127  & 0.17281  \\
2  & 0.1      & 0.1      & 0.11333  & 0.11364  \\
3  & 0.1      & 0.1      & 0.12054  & 0.12093  \\
4  & 0.1      & 0.1      & 0.10584  & 0.10610  \\
5  & 0.1      & 0.099783 & 0.090745 & 0.092299 \\
6  & 0.1      & 0.10905  & 0.13350  & 0.13383  \\
7  & 0.1      & 0.091277 & 0.092572 & 0.092751 \\
8  & 0.1      & 0.099815 & 0.069467 & 0.069445 \\
9  & 0.1      & 0.090815 & 0.0070316& 0.0023358\\
10 & 0.1      & 0.099979 & 0.09571  & 0.095859 \\
\bottomrule
\end{tabular}
\end{minipage}
\end{table*}

\subsubsection{Example} \label{Sec_effect_T_Ex}
Consider the
Laplacian dynamics of $\dot{x}(t)=-Lx(t)$
on the directed graph in Fig. \ref{fig:directed1},
where $L$ is a graph Laplacian with uniform edge weight $0.2$.
Then,
 Algorithm \ref{alg:projgrad} produced VCS and AECS for each state node, as shown in Table \ref{tab:combined}.
 As expected, for small $T$, both scores 
are approximately $1/n=0.1$.
Moreover, as $T$ increases, they converge to constants, indicating that the influence of $T$ saturates in the controllability evaluation.

According to Table \ref{tab:combined}, for sufficiently large $T$, VCS generally assigns higher importance to upstream nodes. As demonstrated in \cite[Section VI-B]{sato2022controllability}, both VCS and AECS tend to assign higher values to upstream nodes in asymptotically stable hierarchical systems.
 In contrast, our study examines Laplacian dynamics, which inherently lack asymptotic stability. Despite this difference, VCS maintains a similar valuation trend, thereby reaffirming the importance of node hierarchy. 
 On the other hand, Table \ref{tab:combined} shows that AECS values are more strongly influenced by node in-degree. This dependence of AECS on in-degree appears to be a consequence of the Laplacian dynamics employed in our work.

\subsubsection{Summary}
In essence, the variations in the scores for different values of
$T$ can be interpreted as follows: For very small
$T$, the short-term dynamics dominate and do not allow for a meaningful differentiation of node importance (hence the uniformity), whereas for larger
$T$, the scores capture the long-term controllability characteristics of the system, providing a more nuanced measure of each node's significance.

\subsection{Convergence analysis of Algorithm \ref{alg:projgrad}} \label{Sec_Convergence}

 In this section, we show that the sequence generated by Algorithm \ref{alg:projgrad} converges linearly to the optimal solution of FTCSP \eqref{prob:unstable} under some assumptions.
That is, we resolve (iv), as mentioned in Section \ref{Subsec_limitation}.

To this end, suppose that $p^{(0)}\in X_T\cap \Delta$ is given, and
define 
\begin{align*}
\mathcal{F}_T^{(0)} := \{p\in {\bb R}^n\mid h_T(p)\leq h_T(p^{(0)})\}\cap (X_T\cap \Delta).
\end{align*}
In the same manner as the proof in \cite[Lemma 1]{sato2022controllability},
we can obtain
    \begin{align*}
        \mathcal{F}_T^{(0)} = \{p\in {\bb R}^n\mid h_T(p)\leq h_T(p^{(0)})\}\cap \Delta
    \end{align*}
and we can conclude that
$\mathcal{F}_T^{(0)}$ is a compact and convex set in ${\bb R}^n$.

\begin{lemma} \label{Lem_strongly_convex}
    The objective function $h_T(p)$ for FTCSP \eqref{prob:unstable} is
    Lipschitz smooth on the set $\mathcal{F}_T^{(0)}$. That is,
    there exists a positive real number $L$ such that
    \begin{align}
        \lambda_{\rm max}(p) \leq L\label{Eq_L_smooth}
    \end{align}
    for any $p\in \mathcal{F}_T^{(0)}$,
    where $\lambda_{\rm max}(p)$ denotes the maximum eigenvalue of $\nabla^2 h_T(p)$.
    Moreover, $h_T(p)$ is strongly convex on $\mathcal{F}_T^{(0)}$. That is, 
    there exists a positive real number $m$ such that
    \begin{align}
        \lambda_{\rm min}(p) \geq m \label{Eq_strong_convex}
    \end{align}
    for any $p\in \mathcal{F}_T^{(0)}$,
    where $\lambda_{\rm min}(p)$ denotes the minimum eigenvalue of $\nabla^2 h_T(p)$.
\end{lemma}
\begin{proof}
Similarly to the proof of \cite[Theorems 1 and 3]{sato2022controllability}, we can show that $h_T(p)$ is strictly convex on $X_T\cap \Delta$.
That is, the Hessian $\nabla^2 h_T(p)$ is positive definite for any $p\in X_T\cap \Delta$.
Thus, $\nabla^2 h_T(p)$ is positive definite on the compact set $\mathcal{F}_T^{(0)}$, since $\mathcal{F}_T^{(0)} \subset X_T\cap \Delta$. 
This implies that there exist $L>0$ and $m>0$ such that \eqref{Eq_L_smooth} and \eqref{Eq_strong_convex} hold for any $p\in \mathcal{F}_T^{(0)}$,
because $\lambda_{\rm max}(p)$ and $\lambda_{\rm min}(p)$ are continuous on the compact set $\mathcal{F}_T^{(0)}$ and the Weierstrass' Theorem applies, as shown in \cite[Proposition A.8]{bertsekas2016nonlinear}. \qed
\end{proof}

Using Lemma \ref{Lem_strongly_convex}, we derive the following linear convergence rate for Algorithm \ref{alg:projgrad}.

\begin{theorem} \label{Thm_linear_convergence}
Suppose that $L$ and $m$ be positive real numbers satisfying \eqref{Eq_L_smooth} and \eqref{Eq_strong_convex} for any $p\in \mathcal{F}_T^{(0)}$, respectively, and suppose that $\{p^{(k)}\}$ and $\{\alpha^{(k)}\}$ were generated by Algorithm \ref{alg:projgrad} with $\varepsilon =0$.
Let $\alpha_{\rm min}$ and $\alpha_{\rm max}$ be constants satisfying $\alpha_{\rm min}>0$ and $\alpha_{\rm max}<2/L$. 
If the sequence $\{\alpha^{(k)}\}$ satisfies
    $\alpha_{\rm min} \leq \alpha^{(k)} \leq \alpha_{\rm max}$
for all $k$,
then $\{p^{(k)}\}$ converges linearly to the optimal solution $p^*$ of FTCSP \eqref{prob:unstable} with the convergence rate
\begin{align*}
    r:=& \max\{|1-\alpha_{\rm max}L|, |1-\alpha_{\rm min}L|,\\
    & |1-\alpha_{\rm max}m|, |1-\alpha_{\rm min}m|\},
\end{align*} meaning that
\begin{align}
    \|p^{(k)}-p^*\| \leq r^k \|p^{(0)}-p^*\| \label{Ineq_linear_convergence}
\end{align}
for all $k$.
\end{theorem}
\begin{proof}
Since $\{p^{(k)}\}$ was generated by Algorithm \ref{alg:projgrad} and
\begin{align*}
    p^* = \Pi_{\Delta}(p^*-\alpha^{(k)} \nabla h_T(p^*))
\end{align*}
holds for all $k$, as explained in \cite[Section 3.3.1]{bertsekas2016nonlinear},
\begin{align}
   & \|p^{(k+1)} -p^*\|^2 \\
   \leq& \|\Pi_{\Delta}(p^{(k)}-\alpha^{(k)}\nabla h_T(p^{(k)}))-\Pi_{\Delta}(p^{*}-\alpha^{(k)}\nabla h_T(p^{*}) )\|^2 \\
   \leq & \|p^{(k)}-\alpha^{(k)}\nabla h_T(p^{(k)})-(p^{*}-\alpha^{(k)}\nabla h_T(p^{*}))\|^2 \\
   \leq & \|p^{(k)}-p^*\|^2 -2\alpha^{(k)}(\nabla h_T(p^{(k)})-\nabla h_T(p^*))^\top (p^{(k)}-p^*) \\
   &+ (\alpha^{(k)})^2 \|\nabla h_T(p^{(k)})-\nabla h_T(p^*)\|^2, \label{Ineq1}
\end{align}
where the second inequality follows from the nonexpansive property of the projection $\Pi_{\Delta}$, as shown in \cite[Proposition 1.1.4]{bertsekas2016nonlinear}.
Moreover, \cite[Proposition B. 5]{bertsekas2016nonlinear} yields
\begin{align}
    & (\nabla h_T(p^{(k)})-\nabla h_T(p^*))^\top (p^{(k)}-p^*) \label{Ineq2}\\
    \geq & \frac{mL}{m+L}\|p^{(k)}-p^*\|^2 +\frac{1}{m+L}\|\nabla h_T(p^{(k)})-\nabla h_T(p^*)\|^2. 
\end{align}
Combining \eqref{Ineq1} and \eqref{Ineq2}, we obtain
\begin{align}
    \|p^{(k+1)} -p^*\|^2
   \leq& \gamma^{(k)} \|p^{(k)}-p^*\|^2 \label{case0} \\
   &+\alpha^{(k)}\beta^{(k)}  \|\nabla h_T(p^{(k)})-\nabla h_T(p^*)\|^2,
\end{align}
where $\gamma^{(k)}:=1 - \frac{2\alpha^{(k)}mL}{m+L}$ and $\beta^{(k)}:=\alpha^{(k)} -\frac{2}{m+L}$.

If $\beta^{(k)}\geq 0$,
\begin{align}
    \beta^{(k)}\|\nabla h_T(p^{(k)})-\nabla h_T(p^*)\|^2 \leq \beta^{(k)}L^2\|p^{(k)}-p^*\|^2. \label{case1}
\end{align}
Here, we have used the fact that \eqref{Eq_L_smooth} holds for any $p\in \mathcal{F}_T^{(0)}$ if and only if $\|\nabla h_T(p_1)-\nabla h_T(p_2)\| \leq L \|p_1-p_2\|$ holds for any $p_1, p_2\in \mathcal{F}_T^{(0)}$, as shown in \cite[Corollary 5.13]{beck2017first}.
If $\beta^{(k)}< 0$,
\begin{align}
    \beta^{(k)}\|\nabla h_T(p^{(k)})-\nabla h_T(p^*)\|^2 \leq \beta^{(k)}m^2\|p^{(k)}-p^*\|^2. \label{case2}
\end{align}
Here, we have used the fact that if \eqref{Eq_strong_convex} holds for any $p\in \mathcal{F}_T^{(0)}$, then $\|\nabla h_T(p_1)-\nabla h_T(p_2)\| \geq m \|p_1-p_2\|$ holds for any $p_1, p_2\in \mathcal{F}_T^{(0)}$, by applying the result in \cite[Theorem 5.24]{beck2017first}.

Using \eqref{case0}, \eqref{case1}, and \eqref{case2}, we obtain an upper bound for $\|p^{(k+1)}-p^*\|^2$:
\begin{align*}
    & \|p^{(k+1)}-p^*\|^2 \\
    \leq & (\gamma^{(k)}+\alpha^{(k)}\max \{\beta^{(k)}L^2,\beta^{(k)}m^2\} )\|p^{(k)}-p^*\|^2 \\
    = & \max \{ (1-\alpha^{(k)}L)^2, (1-\alpha^{(k)}m)^2 \} \|p^{(k)}-p^*\|^2 \\
    \leq & r^2\|p^{(k)}-p^*\|^2,
\end{align*}
which implies that \eqref{Ineq_linear_convergence} holds. \qed
\end{proof}

 Theorem \ref{Thm_linear_convergence}  guarantees that, under the specified conditions, the distance between the iterates $p^{(k)}$ and the optimal solution of FTCSP \eqref{prob:unstable} decreases by at least a fixed factor (less than $1$) at each iteration, ensuring efficient convergence.
To substantiate Theorem \ref{Thm_linear_convergence},
Section~\ref{Sec_convergence_experiment} presents numerical experiments.

\section{Numerical experiments using real-world human brain network data} \label{Sec_numerical}

\begin{figure}[t]
   \centering
   \begin{minipage}[b]{0.48\linewidth}
     \centering
     \includegraphics[width=\linewidth]{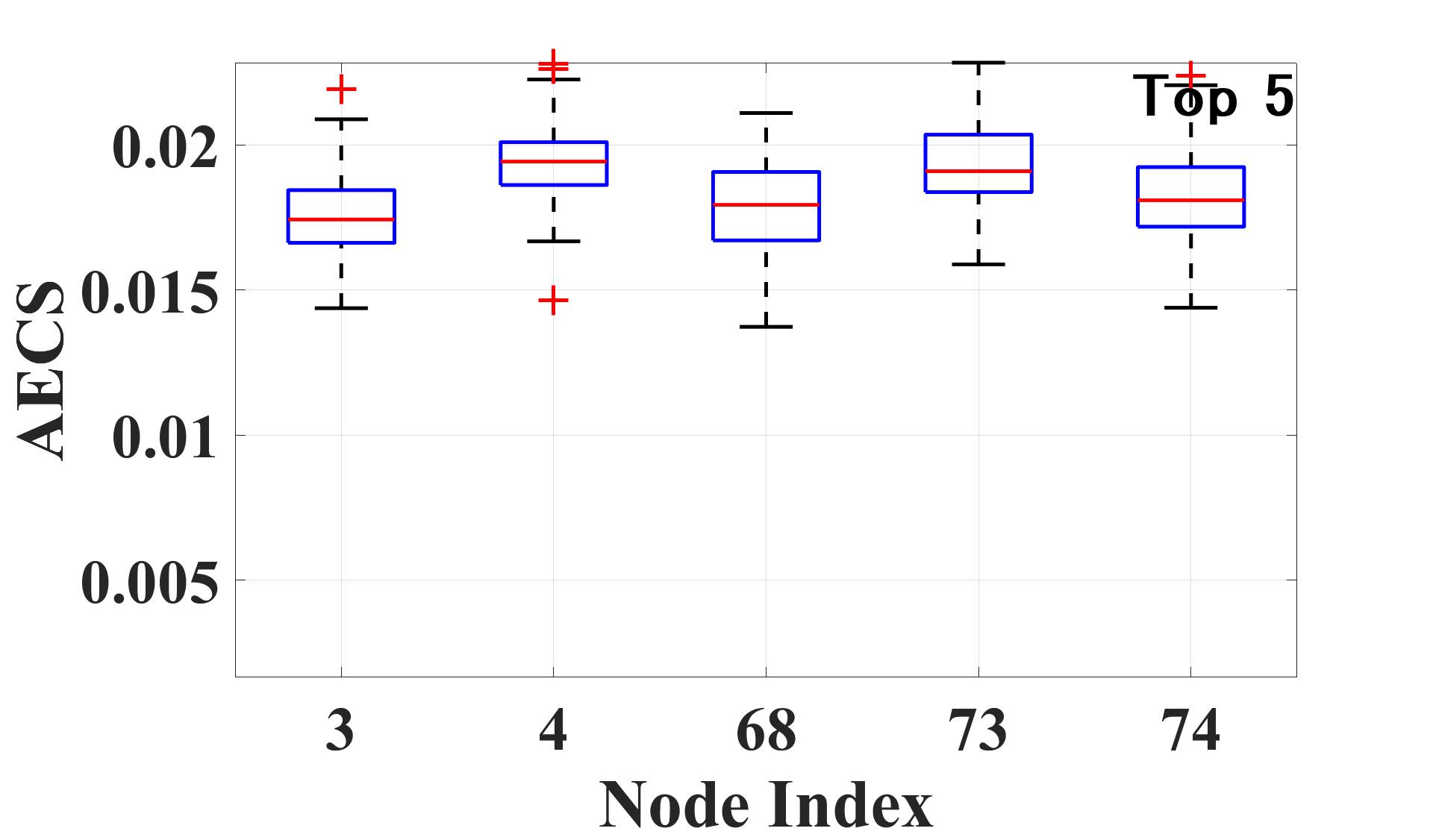}
   \end{minipage}
   \hfill
   \begin{minipage}[b]{0.48\linewidth}
     \centering
     \includegraphics[width=\linewidth]{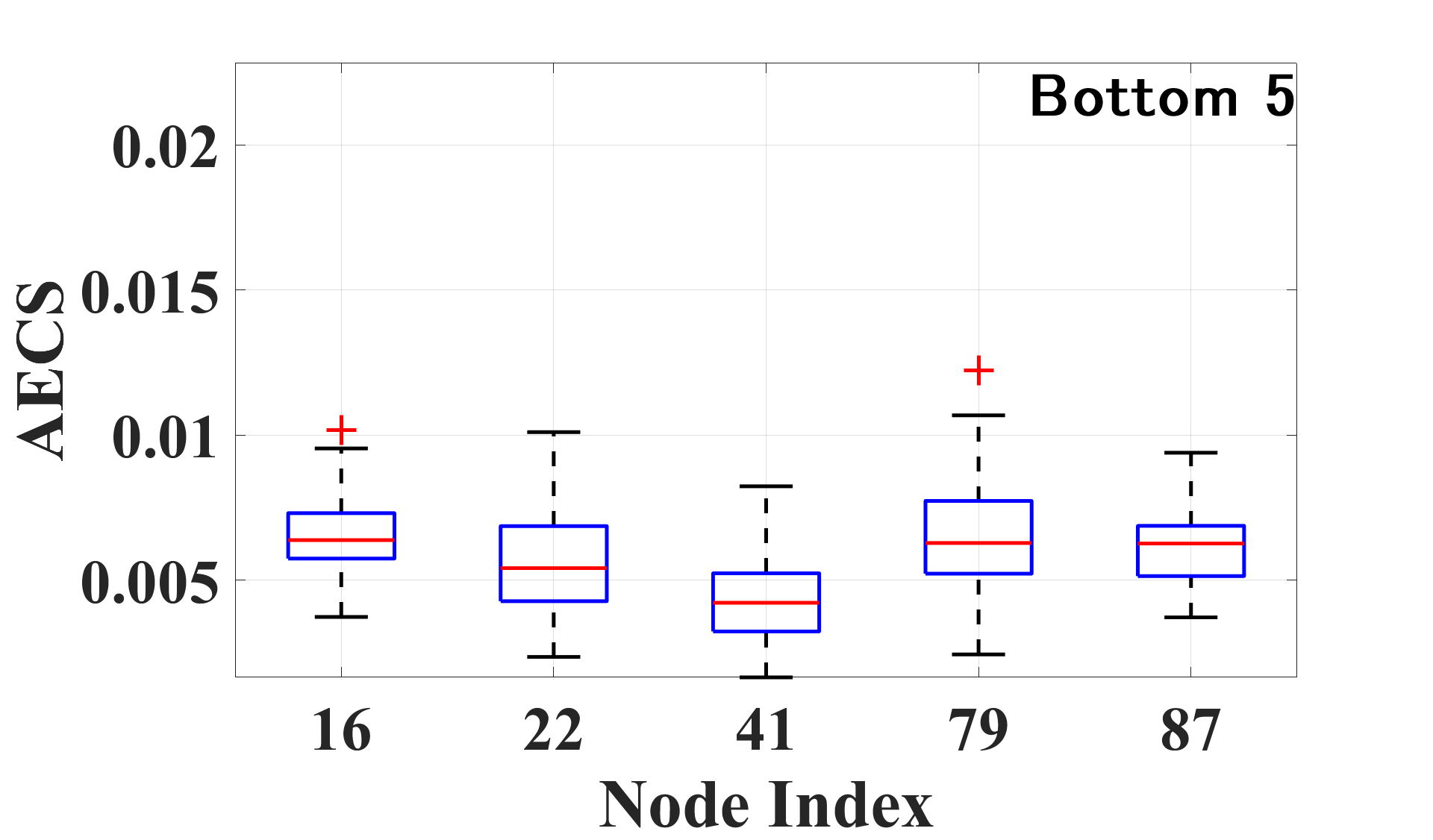}
   \end{minipage}
   \caption{Boxplot of AECS for $T=100$. (Left: Top 5; Right: Bottom 5.)}
   \label{fig:AECS_T100}
\end{figure}

\begin{figure}[t]
   \centering
   \begin{minipage}[b]{0.48\linewidth}
     \centering
     \includegraphics[width=\linewidth]{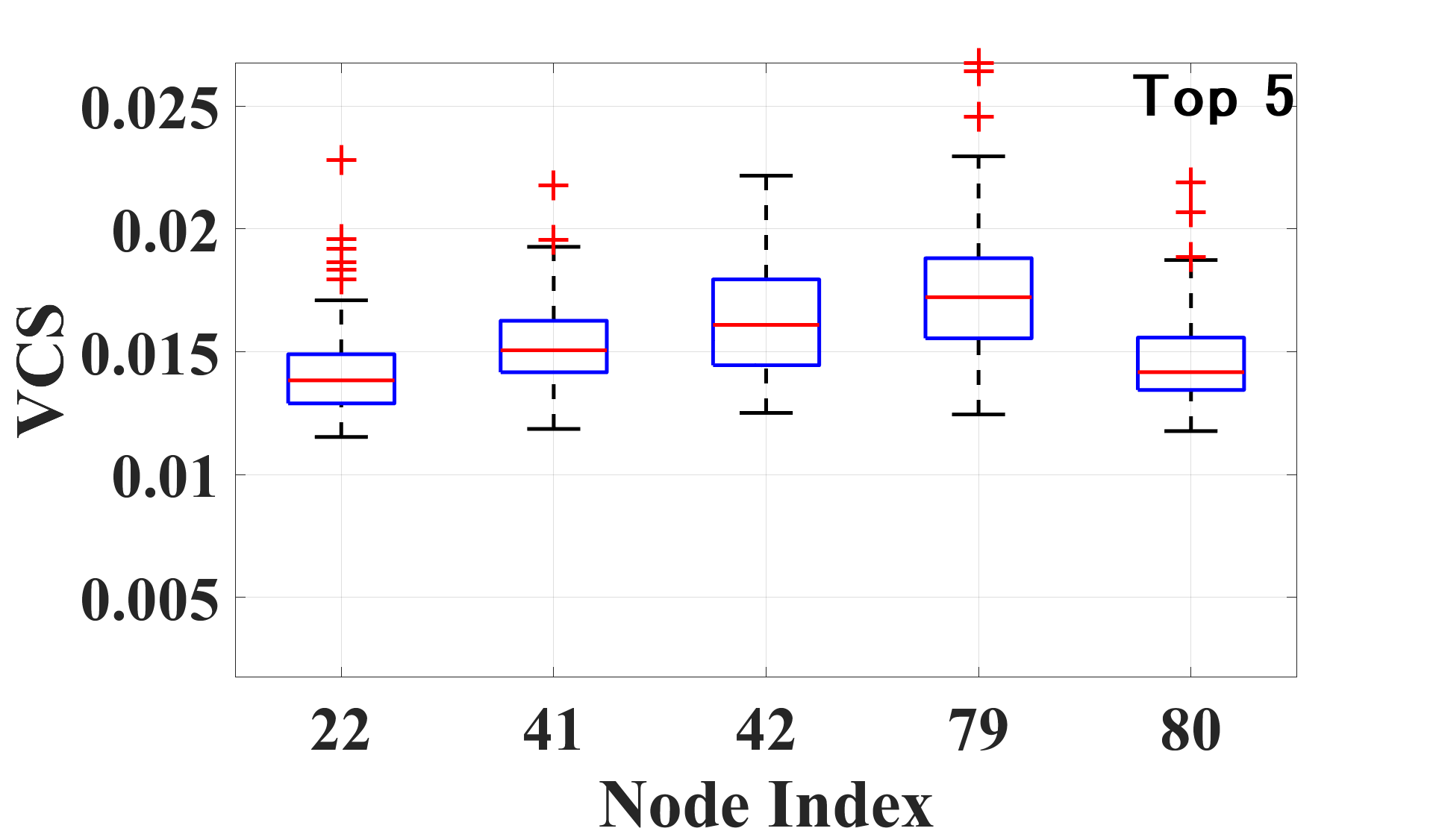}
   \end{minipage}
   \hfill
   \begin{minipage}[b]{0.48\linewidth}
     \centering
     \includegraphics[width=\linewidth]{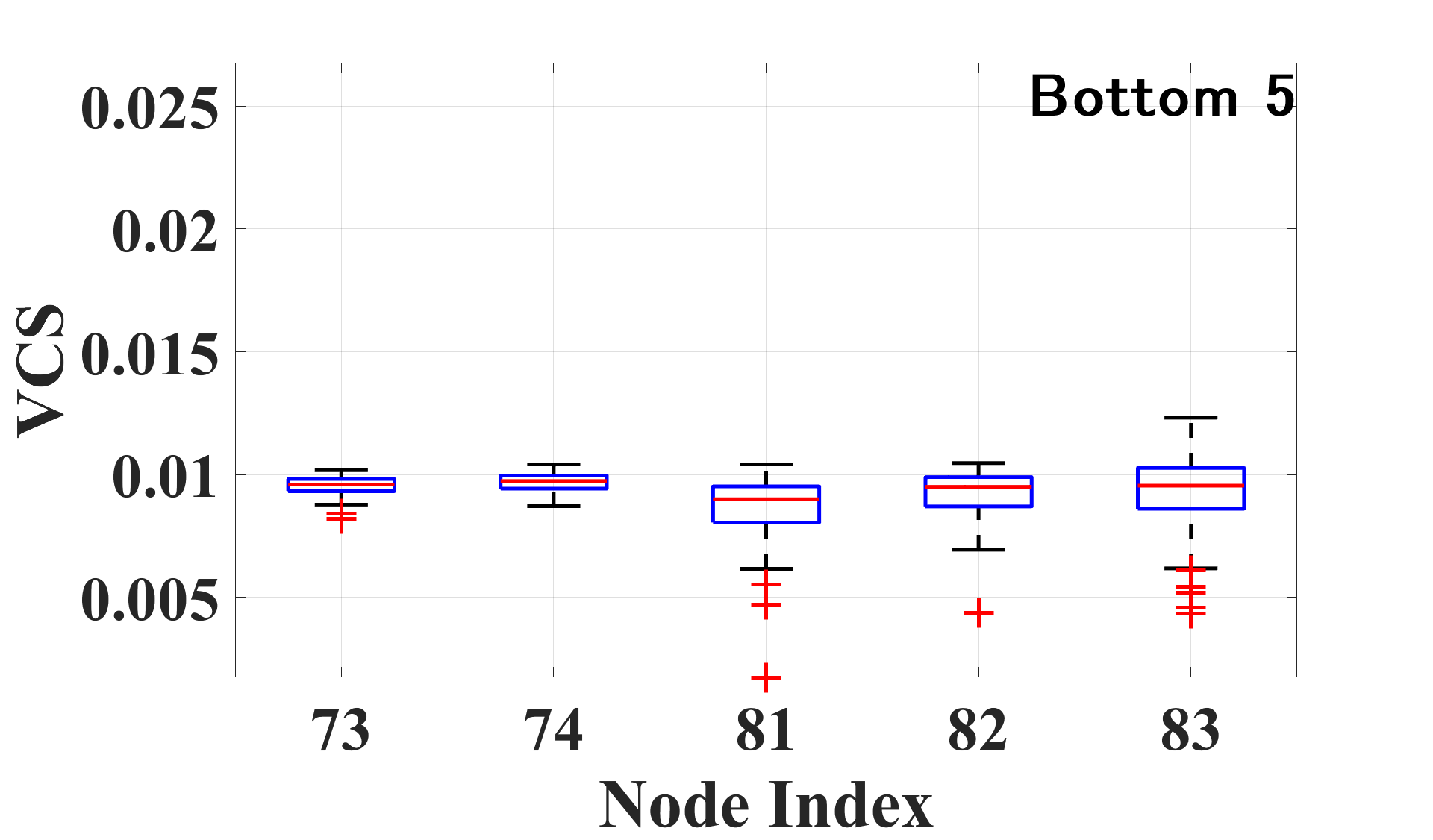}
   \end{minipage}
   \caption{Boxplot of VCS for $T=100$. (Left: Top 5; Right: Bottom 5.)}
   \label{fig:VCS_T100}
\end{figure}

We evaluated VCS and AECS using real-world data from human brain networks, which can be downloaded from \url{https://doi.org/10.17605/OSF.IO/YW5VF},
as provided in
\cite{vskoch2022human}.
This dataset includes connectivity matrix data for $88$ individuals.
Each individual's brain network is represented by a $90\times 90$ matrix, where the matrix element at row $i$ and column $j$ indicates the connectivity probability from the $i$-th region of interest (ROI) to the $j$-th ROI, as defined by the automatic anatomical labeling atlas.
That is, the dataset contains brain networks for $88$ individuals, and each network consists of $90$ nodes corresponding to different brain regions.

We model the individual blood oxygen level dependent (BOLD) signal dynamics as 
the continuous-time system
\begin{align}
   \dot{x}^{(i)}(t) = -\mathcal{L}^{(i)}x^{(i)}(t)\quad (i=1,\ldots, 88), \label{ex_continuous} 
\end{align}
where each component of $x^{(i)}(t)$ denotes the BOLD signal of each ROI at time $t$ for the $i$-th individual.
Here, 
\begin{align*}
 \mathcal{L}^{(i)} := {\rm diag}\left(\sum_{j=1}^{90} \mathcal{A}^{(i)}_{1j}, \ldots, \sum_{j=1}^{90} \mathcal{A}^{(i)}_{90,j} \right) - \mathcal{A}^{(i)}
\end{align*}
is the graph Laplacian, and
$\mathcal{A}^{(i)}\in {\bb R}^{90\times 90}$ denotes the
 transpose of the connectivity matrix for the $i$-th individual.
 This means that we used a directed structural connectivity matrix, which better captures the inherently directional nature of neural connections in the brain.
Thus,
 $\mathcal{L}^{(i)}$ possesses a zero eigenvalue.
However, Theorem \ref{thm:unique_anyA} guarantees that VCS and AECS uniquely exist for almost all $T>0$.
Note that the model \eqref{ex_continuous}  can be regarded as the continuous-time version of a discrete-time model considered in \cite{gu2015controllability}.

Because VCS and AECS depend on finite-time parameter $T$ in optimization problem \eqref{prob:unstable},
we compared VCS and AECS for different values of $T$. 
However, as theoretically proven in Section \ref{Sec_effect_T}, for small values of \(T\) the VCS and AECS lose their interpretative meaning. Moreover, as demonstrated in Section \ref{Sec_effect_T_Ex}, our experiments confirm that as \(T\) increases, both metrics converge to values that are independent of \(T\). Therefore, we focus solely on the results for \(T=100\). Specifically, Fig.~\ref{fig:AECS_T100} illustrates the relationship between AECS and ROI, displaying the top 5 and bottom 5 values in box plots based on 88 individual observations, while Fig.~\ref{fig:VCS_T100} shows the corresponding relationship for VCS. The top 5 and bottom 5 for both AECS and VCS were determined based on the average values across the 90 brain regions for the 88 individuals.

We observed that the range---the difference between the largest and smallest values---of VCS was generally smaller than that of AECS. This observation is consistent with Theorem \ref{Thm_symmetric}, since $\mathcal{L}^{(i)}$ is not symmetric but is close to a symmetric matrix, as reported in \cite{vskoch2022human}.

\begin{remark}
    As for the dynamics, we acknowledge that standard brain activity models incorporate nonlinearities, particularly in the evolution of the BOLD signal \cite{friston2000nonlinear, zeidman2019guide}. However, in our study we adopt a linear Laplacian-based dynamics model as a first-order approximation. This linearization is reasonable when the system operates near a resting state or equilibrium, where perturbations are small. In such regimes, linear models have been shown to capture essential features of the brain’s response while offering significant analytical and computational tractability. This approach has also been supported by prior work in the literature \cite{gu2015controllability, tu2018warnings, pasqualetti2019re, suweis2019brain}.
\end{remark}

\subsection{Convergence rate} \label{Sec_convergence_experiment}

Section~\ref{Sec_Convergence} presents a rigorous mathematical analysis of the convergence properties of Algorithm \ref{alg:projgrad} and proves that it converges linearly under some assumptions. To validate these theoretical findings, we conducted numerical experiments for an individual.

As shown in Fig.~\ref{fig:vcs_aecs_convergence}, Algorithm \ref{alg:projgrad} exhibits linear convergence for both the VCS and AECS cases.
This is evidenced by the fact that when we plot
  $\log \left\| p^{(k)} - p^* \right\|$
on the vertical axis, the result is a straight line with a constant negative slope. In this context, $p^*$ is not the exact optimal solution but is approximated by $p^{(K)}$ obtained after a sufficiently large number of iterations $K$ such that $p^{(K)}$ can be regarded as the converged solution. This numerical experiment thereby validates the mathematical convergence analysis presented in Section~\ref{Sec_Convergence}.

\begin{figure}[tb]
    \centering
   \includegraphics[width=1\linewidth]{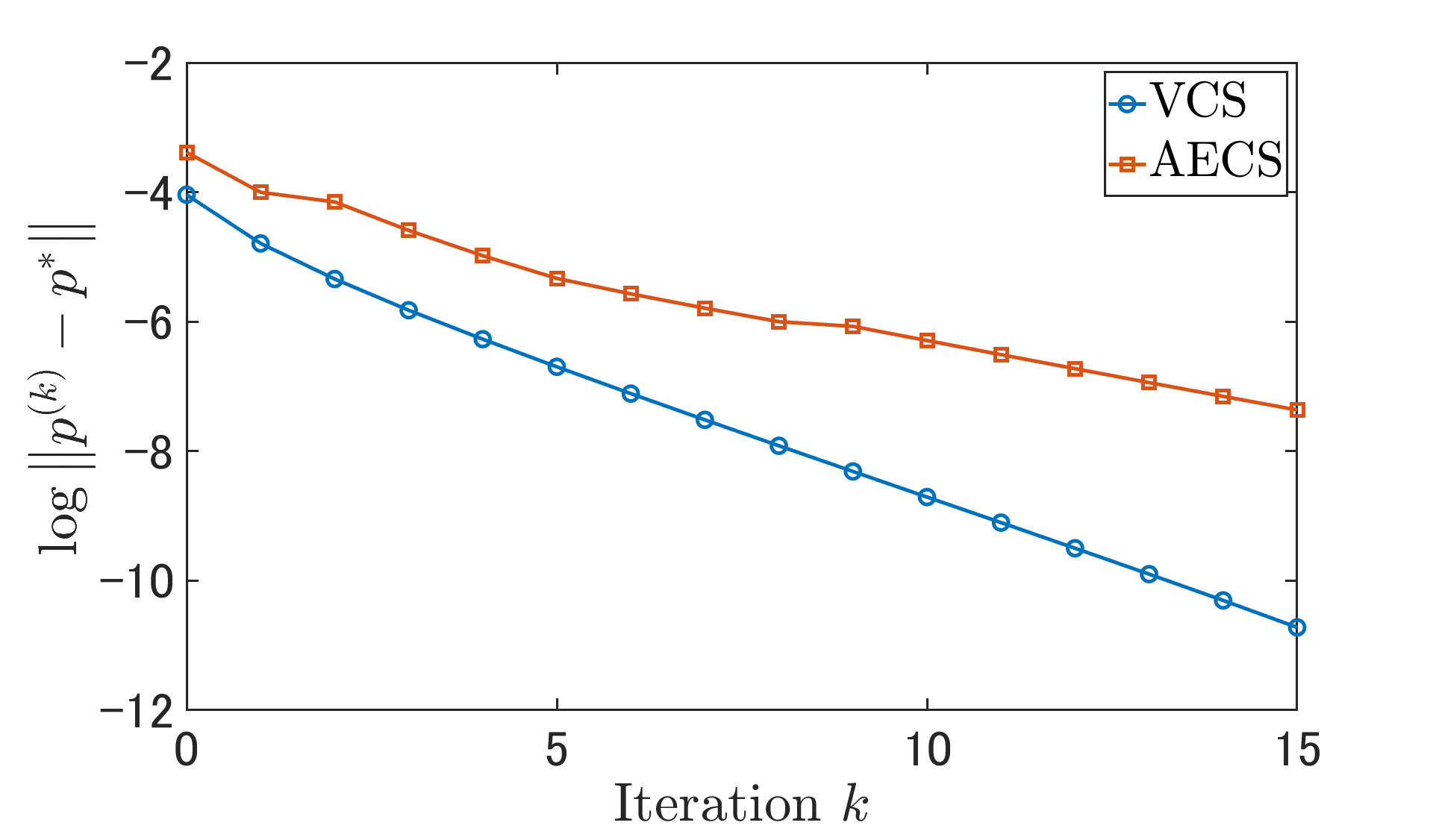}
    \caption{
    Convergence behavior of Algorithm \ref{alg:projgrad}.
    }
    \label{fig:vcs_aecs_convergence}
\end{figure}

\subsection{Diverse Approaches to Brain Network Centrality Analysis}

Tables~\ref{tab:metrics_nodes} and~\ref{tab:bottom5_metrics} provide a comprehensive summary of the top 5 and bottom 5 nodes across multiple centrality measures—including AECS, VCS, Indegree, Outdegree, Betweenness, PageRank, and Average Controllability—along with their corresponding brain regions. 
In table \ref{tab:metrics_nodes}, a check mark ($\checkmark$) in a given column indicates that the corresponding node is ranked among the top 5 according to that specific centrality metric. 
Similarly, table~\ref{tab:bottom5_metrics} uses a check mark ($\checkmark$) to denote nodes ranked among the bottom 5 for each respective metric. These tables serve as the foundation for our analysis, highlighting regions that act as key hubs or serve more peripheral, specialized roles within the brain network.  A detailed examination of these results allows us to appreciate the multidimensional nature of neural connectivity and the distinct information captured by each metric.

According to Figs. \ref{fig:AECS_T100} and \ref{fig:VCS_T100}, our findings reveal that the AECS and VCS metrics yield markedly different hub profiles. The brain regions corresponding to the AECS top 5 nodes and the VCS bottom 5 nodes are primarily involved in higher cognitive functions and motor control \cite{achiron2013superior, boisgueheneuc2006functions, cavanna2006precuneus, hu2016right, yu2013enhanced}. 
Conversely, the brain regions corresponding to the AECS bottom 5 nodes and the VCS top 5 nodes are predominantly linked to sensory processing and emotional regulation \cite{hartwigsen2019functional, koelsch2014brain, ledoux2003emotional, soudry2011olfactory}. 

Moreover, when applying traditional graph-theoretical metrics such as Indegree, Outdegree, Betweenness, and PageRank \cite{bloch2023centrality, rodrigues2019network} to the matrix $\mathcal{A}^{(i)}$ $(i=1,\ldots, 88)$, regions like the Precuneus, Putamen, and Thalamus consistently rank highly. Each of these measures assesses a different aspect of centrality: Indegree and Outdegree reflect local connectivity; Betweenness quantifies the role of a node in facilitating global information flow; and PageRank indicates a node's overall influence within the network. The recurrent prominence of these regions across various metrics underscores their vital role as integrative hubs that mediate and disseminate information across the brain.

The fact that VCS and Average Controllability yield very similar top 5 nodes but markedly different bottom 5 nodes suggests that these two metrics converge in identifying the most influential or ``hub" regions of the network, yet diverge in their characterization of the less central regions. 
Here, for each individual $i\in \{1,\ldots, 88\}$ and for each brain region $j\in \{1,\ldots, 90\}$, Average Controllability (introduced in \cite{summers2015submodularity}) is defined as 
\begin{align*}
    {\rm tr}(\mathcal{W}^{(i)}_{j}),
\end{align*}
where 
\begin{align*}
\mathcal{W}^{(i)}_{j} := \int_0^{T} \exp (-\mathcal{L}^{(i)}t)e_je_j^\top \exp (-(\mathcal{L}^{(i)})^\top t)\, {\rm d}t,
\end{align*}
which corresponds to the finite-time controllability Gramian of the system with only one input
\begin{align}
    \dot{x}^{(i)}(t) = -\mathcal{L}^{(i)}x^{(i)}(t) + e_j u(t). \label{eq_brain_input}
\end{align}
This means that the external input is applied only to the $j$-th brain region of the $i$-th individual, and no input is applied to other regions.

Table~\ref{Table:correlation} summarizes the average correlation coefficients and standard deviations between the AECS/VCS metrics and various node-level network measures—including Indegree, Outdegree, Betweenness, PageRank, and Average Controllability—across $88$ individuals. Notably, the AECS metric exhibits moderately strong positive correlations with the traditional centrality measures, with average correlation coefficients ranging from approximately $0.64$ (for Indegree and Outdegree) to $0.68$ (for Betweenness and PageRank). In contrast, AECS shows a moderately strong negative correlation (average coefficient of $-0.60$) with Average Controllability. Conversely, the VCS metric demonstrates moderate negative correlations with Indegree, Outdegree, Betweenness, and PageRank (coefficients around $-0.26$ to $-0.30$), while it is very strongly positively correlated with Average Controllability (average coefficient of $0.84$). These findings underscore the complementary roles of AECS and VCS.

\begin{table*}[t]
    \centering
    \caption{Top 5 Nodes and Their Corresponding Brain Regions for AECS, VCS, Indegree, Outdegree, Betweenness, PageRank, and Average Controllability.}
    \begin{tabular}{ccccccccc}
        \toprule
        Node Index & Brain Region & AECS & VCS & Indegree & Outdegree & Betweenness & PageRank & Average Controllability \\
        \midrule
         3  & Left Superior Frontal Gyrus   & $\checkmark$ &             &              &              &              &             &  \\
         4  & Right Superior Frontal Gyrus  & $\checkmark$ &             &              &              &              &             &  \\
        22  & Right Olfactory Cortex        &             & $\checkmark$ &              &              &              &             & $\checkmark$ \\
        41  & Left Amygdala                 &             & $\checkmark$ &              &              &              &             & $\checkmark$ \\
        42  & Right Amygdala                &             & $\checkmark$ &              &              &              &             & $\checkmark$ \\
        67  & Left Precuneus                &             &             & $\checkmark$ & $\checkmark$ & $\checkmark$ & $\checkmark$ &  \\
        68  & Right Precuneus               & $\checkmark$ &             & $\checkmark$ &              &              & $\checkmark$ &  \\
        73  & Left Putamen                  & $\checkmark$ &             &              & $\checkmark$ & $\checkmark$ &             &  \\
        74  & Right Putamen                 & $\checkmark$ &             & $\checkmark$ & $\checkmark$ & $\checkmark$ & $\checkmark$ &  \\
        77  & Left Thalamus                 &             &             & $\checkmark$ & $\checkmark$ & $\checkmark$ & $\checkmark$ &  \\
        78  & Right Thalamus                &             &             & $\checkmark$ & $\checkmark$ & $\checkmark$ & $\checkmark$ &  \\
        79  & Left Heschl's Gyrus           &             & $\checkmark$ &              &              &              &             & $\checkmark$ \\
        80  & Right Heschl's Gyrus          &             & $\checkmark$ &              &              &              &             &  \\
        87  &  Left Medial Temporal Pole                             &             &             &              &              &              &             & $\checkmark$ \\
        \bottomrule
    \end{tabular}
    \label{tab:metrics_nodes}
\end{table*}

\begin{table*}[t]
    \centering
    \caption{Bottom 5 Nodes and Their Corresponding Brain Regions for AECS, VCS, Indegree, Outdegree, Betweenness, PageRank, and Average Controllability.}
    \begin{tabular}{ccccccccc}
        \toprule
        Node Index & Brain Region & AECS & VCS & Indegree & Outdegree & Betweenness & PageRank & Average Controllability \\
        \midrule
         4  &   Right Superior Frontal Gyrus                      &       &       &          &           &           &         & $\checkmark$ \\
        16  & Right IFG (p. Orbitalis) & $\checkmark$ &       &          &           &           &         &  \\
        21  & Left Olfactory Cortex    &       &       & $\checkmark$ & $\checkmark$ &       & $\checkmark$ &  \\
        22  & Right Olfactory Cortex   & $\checkmark$ &       & $\checkmark$ & $\checkmark$ & $\checkmark$ & $\checkmark$ &  \\
        27  & Left Rectal Gyrus        &       &       & $\checkmark$ &           &           & $\checkmark$ &  \\
        28  & Right Rectal Gyrus       &       &       & $\checkmark$ & $\checkmark$ &         & $\checkmark$ &  \\
        40  & Right ParaHippocampal Gyrus &     &       &          & $\checkmark$ & $\checkmark$ &         &  \\
        41  & Left Amygdala            & $\checkmark$ &       & $\checkmark$ &           & $\checkmark$ & $\checkmark$ &  \\
        51  &   Left Middle Occipital Gyrus                    &       &       &          &           &           &         & $\checkmark$ \\
        54  & Right Inferior Occipital Gyrus & &       &          & $\checkmark$ & $\checkmark$ &         &  \\
        67  &  Left Precuneus                       &       &       &          &           &           &         & $\checkmark$ \\
        68  &   Right Precuneus                      &       &       &          &           &           &         & $\checkmark$ \\
        70  & Paracentral Lobule       &       &       &          &           & $\checkmark$ &         &  \\
        73  & Left Putamen             &       & $\checkmark$ &        &           &           &         &  \\
        74  & Right Putamen            &       & $\checkmark$ &        &           &           &         &  \\
        79  & Left Heschl's Gyrus      & $\checkmark$ &       &          &           &           &         &  \\
        81  & Left Superior Temporal Gyrus &   & $\checkmark$ &       &           &           &         &  \\
        82  & Right Superior Temporal Gyrus &  & $\checkmark$ &       &           &           &         &  \\
        83  & Left Temporal Pole       &       & $\checkmark$ &        &           &           &         &  \\
        85  &  Left Middle Temporal Gyrus                       &       &       &          &           &           &         & $\checkmark$ \\
        87  & Left Medial Temporal Pole & $\checkmark$ &       &        &           &           &         &  \\
        \bottomrule
    \end{tabular}
    \label{tab:bottom5_metrics}
\end{table*}

In summary, we identified tendencies and correlations of AECS and VCS, as summarized in Table \ref{table:summary}.

\begin{table}[t]
    \centering
    \caption{Average and Standard Deviation of Correlation Coefficients between AECS/VCS and Node-Level Network Metrics (Indegree, Outdegree, Betweenness, PageRank, and Average Controllability) across 88 Individuals.}
    \label{Table:correlation}
    \begin{tabular}{|c|c|c|}
        \hline
        & \textbf{Average} & \textbf{Standard Deviation}  \\
        \hline
        \textbf{AECS and Indegree} & $0.64372$ & $0.047636$ \\
        \hline
        \textbf{AECS and Outdegree} & $0.64043$ & $0.046415$  \\
        \hline
        \textbf{AECS and Betweenness} & $0.67655$ & $0.055664$ \\
        \hline
        \textbf{AECS and PageRank} & $0.64493$ & $0.048493$  \\
        \hline
        \textbf{AECS and Ave. Con.} & $-0.6036$ & $0.067408$  \\
        \hline
        \textbf{VCS and Indegree} & $-0.29601$ & $0.076079$  \\
        \hline
        \textbf{VCS and Outdegree} & $-0.26841$ & $0.075018$  \\
        \hline
        \textbf{VCS and Betweenness} & $-0.25723$ & $0.057868$  \\
        \hline
        \textbf{VCS and PageRank} & $-0.29478$ & $0.075393$  \\
         \hline
        \textbf{VCS and Ave. Con.} & $0.83696$ & $0.049695$  \\
        \hline
    \end{tabular}
\end{table}

\subsection{Clinical Implications of Novel Brain Network Metrics}

\subsubsection{VCS and Sensory-Affective Integration: Novel Insights and Clinical Implications}

Our analysis using the VCS metric in healthy individuals reveals that regions such as Heschl's gyrus, the amygdala, and the right olfactory cortex exhibit high centrality, underscoring their essential role in maintaining the integrity of sensory processing networks \cite{ledoux2003emotional, koelsch2014brain}. Because previous studies have shown that sensory processing regions exhibit altered network topology in autism spectrum disorders \cite{rudie2013altered, keown2017network}, using VCS may help identify deviations from the normal centrality patterns seen in healthy individuals.
In other words, because VCS assigns high values to these key sensory-affective regions in healthy brains, applying this metric in clinical populations could reveal abnormal centrality patterns, thereby enhancing our understanding of their neurophysiological underpinnings and informing targeted interventions.

This observation aligns well with established neurophysiological principles that emphasize the central role of sensory cortices in processing external stimuli and the importance of emotional evaluation in generating appropriate behavioral responses (e.g., \cite{ledoux2000emotion, pessoa2008relationship}).
Notably, VCS has the potential to yield novel insights because it diverges significantly from traditional centrality measures. While conventional metrics primarily capture aspects of network integration and information transfer, VCS is more closely related to network controllability. This unique perspective suggests that VCS may uncover previously overlooked patterns in brain network organization, paving the way for new discoveries in understanding sensory processing and its related disorders.

\subsubsection{AECS in Healthy Subjects: Implications for Parkinson's Disease and Beyond}

In healthy individuals, the high AECS scores observed in regions such as the Putamen \cite{yu2013enhanced}, Superior Frontal Gyrus \cite{boisgueheneuc2006functions, hu2016right}, and Precuneus \cite{cavanna2006precuneus} indicate that these areas play central roles in the brain's network. For example, the Putamen is primarily involved in motor control, the Superior Frontal Gyrus in higher-order cognitive control, and the Precuneus in spatial awareness and integrated information processing.

These findings can be related to studies on motor dysfunction in neurodegenerative diseases like Parkinson's disease. In Parkinson's disease, disruptions in motor circuits---particularly in the Putamen \cite{yu2013enhanced}---are directly associated with motor symptoms. In response, cognitive control regions such as the Superior Frontal Gyrus may become more active as a compensatory mechanism to help manage these deficits.

Therefore, using the AECS metric to quantify the centrality of these regions in healthy subjects provides a baseline against which changes in network structure and compensatory mechanisms in neurodegenerative conditions can be assessed. This approach holds promise for the development of new diagnostic markers, the assessment of disease progression, and the monitoring of treatment efficacy.

\subsection{Further Discussions}

In this section, we focus on the Right Superior Frontal Gyrus and the Left Heschl’s Gyrus, as these regions were identified as the most important by AECS and VCS, respectively, and are thus selected for further detailed analysis.

Fig.~\ref{fig:eigenvalues} illustrates the first $50$ (out of $90$) eigenvalues of the controllability Gramian 
$W^{(16)}_4$ for the Right Superior Frontal Gyrus (blue circles) and  $W^{(16)}_{79}$ for
 the Left Heschl’s Gyrus (orange squares). In both cases, the eigenvalues decay rapidly, reaching extremely small magnitudes (on the order of 
$10^{-15}$ to $10^{-20}$) by approximately the $20$th eigenvalue, indicating that many directions in the state space are only weakly controllable.
Such extremely small values are most likely due to round-off errors inherent in finite-precision arithmetic rather than representing genuinely nonzero eigenvalues. Therefore, we interpret these eigenvalues as effectively zero. This interpretation is consistent with the expected theoretical properties of the matrix and is a common occurrence in numerical computations.

Moreover, Fig.~\ref{fig:eigenvalues} indicates that Volumetric Control Energy (VCE) and Average Control Energy (ACE) centralities introduced in \cite{summers2015submodularity} do not fully capture the importance of each brain region from the viewpoint of controllability, as defined below: \begin{itemize}
    \item 
    For each individual $i\in \{1,\ldots, 88\}$ and for each brain region $j\in \{1,\ldots, 90\}$,
VCE centrality has been defined as
\begin{align}
\sum_{s=1}^{k_j^{(i)}}\log \lambda_s(\mathcal{W}^{(i)}_j), \label{eq:vce}
\end{align}
where $k_j^{(i)}:={\rm rank}\,\mathcal{W}^{(i)}_j$, and $\lambda_1(\mathcal{W}^{(i)}_j),\ldots, \lambda_{k_j^{(i)}}(\mathcal{W}^{(i)}_j)$ denote the eigenvalues of $\mathcal{W}^{(i)}_j$ satisfying $\lambda_1(\mathcal{W}^{(i)}_j)\geq \cdots \geq \lambda_{k_j^{(i)}}(\mathcal{W}^{(i)}_j)>0$.
Here,
$k_j^{(i)}$ means the controllable subspace dimension of system \eqref{eq_brain_input}.
Since the VCE value is related to the volume 
of the ellipsoid spanned by the reachable subspace, a larger value indicates better controllability.

\item
    For each individual $i\in \{1,\ldots, 88\}$ and for each brain region $j\in \{1,\ldots, 90\}$,
ACE centrality 
 has been defined as
 \begin{align*}
     {\rm tr}\left((W_j^{(i)})^\dagger\right),
 \end{align*}
where $(W_j^{(i)})^\dagger$ is the pseudo inverse of $W_j^{(i)}$.
This quantity is related to the average minimum input energy required to steer the system onto the unit sphere of the reachable subspace, implying that lower ACE values indicate better controllability.
\end{itemize}

It has been pointed out in \cite{sato2022controllability} that VCE and ACE centralities may not adequately reflect the importance of each state node in terms of controllability for general settings---not limited to brain networks (see Remarks 2 and 3, and Section VI-B in \cite{sato2022controllability}).
This conclusion arises from the fact that these metrics are computed based on the controllability Gramian for a system with only one input. Moreover, related discussions can be found in \cite{tu2018warnings, pasqualetti2019re, suweis2019brain}. In fact, these works specifically address the scenario where a single input is applied to brain networks and demonstrate that the smallest eigenvalue of the controllability Gramian becomes extremely small. This observation implies that using VCE and ACE as centrality measures can be problematic in general, since the resulting metrics may not reliably capture node importance beyond the particular real data analyzed in our study.

\begin{figure}[tb]
    \centering
   \includegraphics[width=1\linewidth]{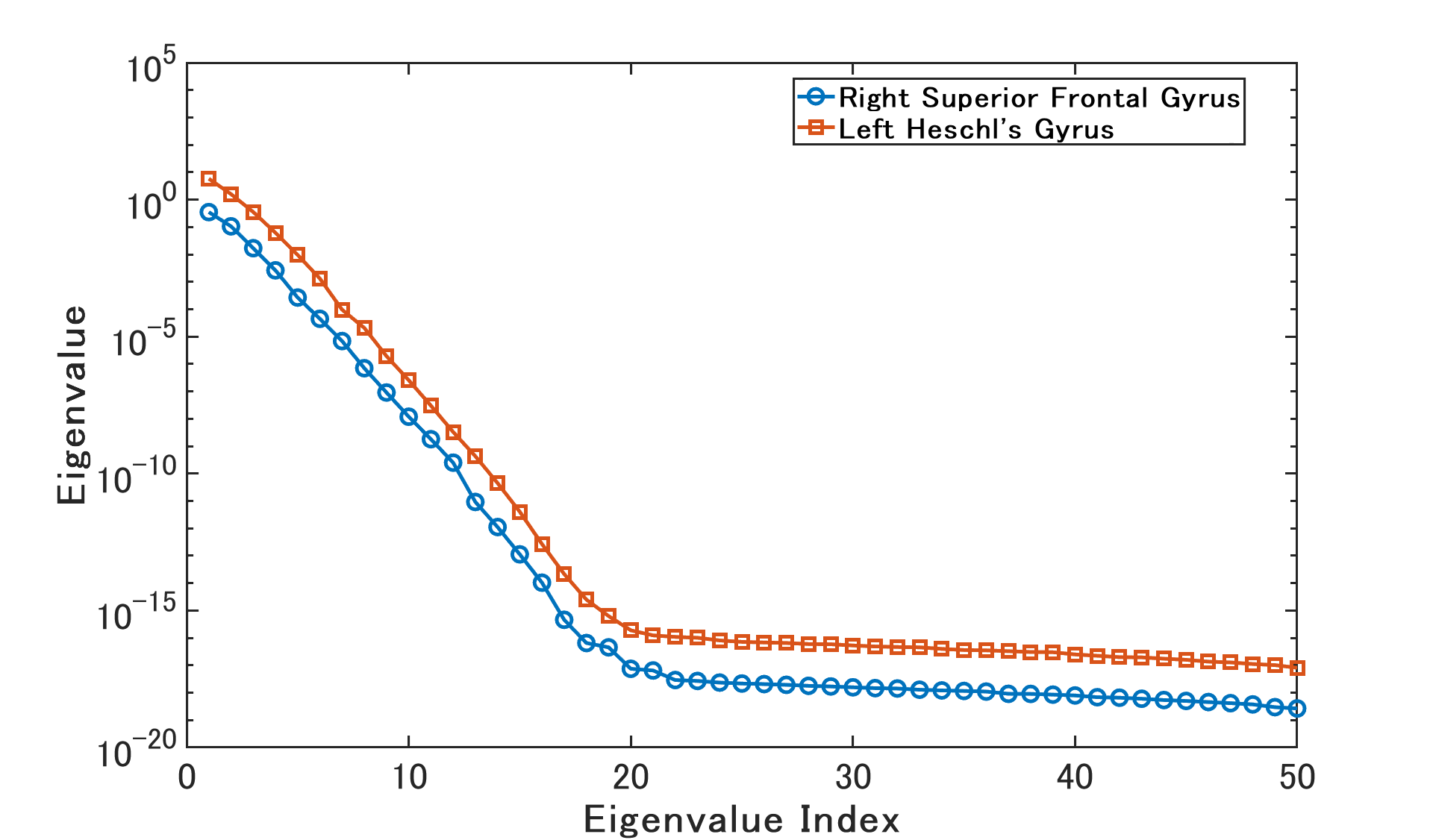}
    \caption{
    Comparison of the first $50$ eigenvalues of the controllability Gramian $W^{(16)}_4$ for the Right Superior Frontal Gyrus and  $W^{(16)}_{79}$ for the Left Heschl’s Gyrus, highlighting their rapid decay.
    }
    \label{fig:eigenvalues}
\end{figure}

\section{Concluding Remarks} \label{Sec_conclusion}

\subsection{Summary}
This paper has significantly advanced the understanding of the Volumetric Controllability Score (VCS) and Average Energy Controllability Score (AECS) within the framework of network systems control. We have established that, for any linear time-invariant system, VCS and AECS are unique 
for almost all specified time parameters $T$, thereby extending the applicability of these measures to a much broader class of systems. Moreover, we demonstrated that the underlying system yields distinct differences between VCS and AECS for symmetric matrices, while both measures coincide for skew-symmetric matrices.
In addition, we investigated the dependence of these scores on time parameter $T$
and proved that when $T$ is extremely small, both VCS and AECS become essentially uniform---a finding supported by numerical experiments. We also provided a detailed convergence analysis of the algorithm used for computing these measures, showing that under several assumptions it converges linearly, which was further validated through experiments.
Finally, our evaluation on brain networks modeled via Laplacian dynamics using real data revealed contrasting evaluation trends: VCS tends to assign higher values to regions associated with sensory processing and emotional regulation, whereas AECS favors regions linked to cognitive function and motor control. These differences, along with their correlations to traditional centrality measures, highlight the complementary perspectives of VCS and AECS in assessing network controllability.

\subsection{Discrepancies Between VCS and AECS: Open Questions}
 Our theoretical analyses and empirical studies on real-world human brain network data (as detailed in Table 1 and Section IV) have revealed significant differences between VCS and AECS.  
 However, the precise reasons behind the discrepancies remain unclear. 
In fact, this comparison hinges on the assumption that BOLD‐signal propagation can be approximated by the linear Laplacian system \eqref{ex_continuous}. While this simplification enabled tractable computation of VCS and AECS on high‐resolution connectomes, true BOLD dynamics emerge from nonlinear hemodynamic and neuronal processes (e.g., \cite{friston2000nonlinear, zeidman2019guide}). Consequently, the pronounced divergences we observed may partly reflect limitations of the Laplacian approximation rather than intrinsic differences between the scores.

\subsection{Extension to Time-Varying Networks}

 Our analysis framework is not limited to static networks but can also be generalized to networks with time-varying structures. Although the mathematical treatment of time-varying structures is inherently more challenging---necessitating additional assumptions for proving the uniqueness of the controllability scores and minor modifications to Algorithm \ref{alg:projgrad}---our additional research (see \cite{umezu2025controllability}) demonstrates that the core findings of this work remain applicable. Notably, numerical experiments in \cite{umezu2025controllability}
indicate that the controllability score may vary between time-varying and static networks.

However, we acknowledge that the application of our approach to real-world datasets with time-varying structures has not yet been performed, and this remains an important avenue for future research.

\subsection{Extension to Nonlinear Systems}
While our VCS and AECS are defined using the controllability Gramian for linear systems---which makes a direct extension to nonlinear systems nontrivial---they may still offer valuable insights in the nonlinear context. In many practical scenarios, such as brain networks that are inherently nonlinear, the system dynamics can be approximated by linear models near equilibrium points. Thus, VCS and AECS can be applied locally to capture the controllability properties of nonlinear systems around these operating points. This potential for local application provides a promising avenue for extending the utility of our approach to more general, nonlinear settings, which we plan to explore in future work.

\subsection{Generalizability to Diverse Network Domains}
In addition to their demonstrated utility in brain network analysis, our proposed metrics---VCS and AECS---hold promise for application in a variety of other networked systems. Future work could explore their use in domains such as social and infrastructure networks, thereby validating their generalizability. For example, in social networks, VCS and AECS could be employed to identify influential individuals or communities by quantifying the controllability and integration of nodes. In infrastructure networks, such as transportation or communication systems, these metrics could help assess the resilience and vulnerability of critical nodes, guiding strategies for optimization and intervention. Investigating these diverse domains may yield valuable insights into the universal principles underlying complex network dynamics, further demonstrating the broad applicability of our approach.


\section*{Acknowledgment}
This work was supported by the Japan Society for the Promotion of Science KAKENHI under Grant  23K03899. 



\bibliographystyle{IEEEtran}
\bibliography{main.bib}




%




\end{document}